\documentclass[12pt]{article}

\usepackage{amsmath, amsfonts, amssymb}
\usepackage{mathrsfs}
\usepackage{theorem}
\usepackage{bm}
\usepackage[dvipdf]{graphicx}

\RequirePackage[colorlinks,citecolor=blue,urlcolor=blue,bookmarksopen]{hyperref}

\newtheorem{theorem}{Theorem}[section]

\newtheorem{lemma}[theorem]{Lemma}
\newtheorem{definition}[theorem]{Definition}
\newtheorem{assumption}[theorem]{Assumption}

\newtheorem{example}[theorem]{Example}

\newcommand{\ra}{\rightarrow}

\newcommand{\RNum}[1]{\uppercase\expandafter{\romannumeral #1\relax}}

\def\L{\mathscr{L}}

\def\R{\mathbb R}

\def\N{\mathbb N}

\def\B{\mathscr B}

\def\da{\downarrow}
\def\F{\mathscr F}
\def\d{\text{\rm{d}}}
\def\E{\mathbb E}

\def\e{\text{\rm{e}}}

\def\de{\delta}

\def\pb{\mathscr{P}}

\def\W{\mathbb{W}}

\newenvironment{proof}{{\noindent\it Proof.}\ }{\hfill $\square$\par}

\numberwithin{equation}{section}

\allowdisplaybreaks
\begin{document}
	\title{Euler-Maruyama method for distribution dependent stochastic differential equation driven by multiplicative fractional Brownian motion}
	
	\author{ Guangjun Shen, Jiangpeng Wang\footnote{Corresponding author. This research is supported by the National Natural Science Foundation of China (12571160).}, and Xuekang Zhang}
	
	\maketitle
	
	\begin{abstract}
		In this paper, we establish the propagation of chaos and Euler-Maruyama method of DDSDE driven by multiplicative fractional Brownian motion with Hurst parameter $H\in (\frac{\sqrt{5}-1}{2},1)$. We have not only obtained an upper bound for the error of the Euler-Maruyama method but also verified the correctness of this result via systematic numerical simulation experiments.
		\vskip.2cm \noindent {\bf Keywords}: Distribution dependent SDE; fractional Brownian motion; Propagation of chaos; Euler-Maruyama method
		\vskip.2cm \noindent {\bf Mathematics Subject Classification:}
	\end{abstract}

	\section{Introduction}
	Stochastic differential equations (SDEs) are the stochastic generalization of deterministic differential equations (ODEs/PDEs), providing a mathematical description for dynamic systems with random perturbations. Their development stems from the fact that many real-world systems cannot be accurately characterized by deterministic models, necessitating the introduction of "stochastic noise" to capture uncertainty. The most typical noise source is Brownian motion. At this point, SDE can be described as
	\begin{equation}
		\d X_t=b(X_t)\d t+\sigma(X_t)\d B_t.
	\end{equation}
	Afterwards, Mckean \cite{Mckean} further promoted this work and accomplished the pioneering work of distribution dependent SDE (also known as McKean-Vlasov SDE). Its core idea is to describe the group effect, that is, the evolution of individuals is influenced by the distribution of the entire group's state. Compared with traditional SDE, DDSDE's coefficients not only depend on the state $X_t$, but also on the state distribution of $X_t$, and have the following form
	\begin{equation}\label{eq-1.2}
		\d X_t=b(X_t, \L_{X_t})\d t+\sigma(X_t, \L_{X_t})\d B_t,
	\end{equation}
	where $\L_{X_t}$ stands for the law of $X_t$. This class of equations gained increasing attention following Dawson?s foundational work \cite{Da} and the development of the Lion's derivative with respect to the measure variables \cite{Li}. DDSDEs are  widely used to model random phenomena across various scientific domains, including physics, biology, engineering, and neural activities, such as \cite{BFFT,BCC,BFT,Carmona2} and references therein.
	Another important feature of DDSDE is its inseparable relationship with the limit behavior of a class of interacting particle systems
	\begin{equation}\label{eq-1.3}
		\d X_t^{i,N}= b(X_t^{i,N},\widehat{\mu}_t^N)\d t+\sigma(X_t^{i,N},\widehat{\mu}_t^N)\d B_t^i,~~ i=1,\cdot\cdot\cdot,N,
	\end{equation}
	where $\widehat{\mu}_t^N$ means the empirical distribution corresponding to $X_t^{1,N},\cdot\cdot\cdot, X_t^{N,N},$ namely,
	$
	\widehat{\mu}_t^N:=\frac{1}{N}\sum^N_{i=1}\delta_{X_t^{i,N}},
	$
	where $\delta_x$ denotes the Dirac measure at point $x$.
	$B^1:=(B^1_t)_{t\geq0}, \cdot\cdot\cdot, B^N:=(B^N_t)_{t\geq0}$ are independent $d$-dimensional Brownian motions on some complete filtered probability space. When the number of particles N approaches infinity, the mean field limit of the following particle system \eqref{eq-1.3} under mean field interaction will tend towards \eqref{eq-1.2}, which is known as propagation of chaos. There are numerous articles on DDSDE research. Huang and  Wang \cite{3} studied the Well-posedness of DDSDE driven by non-degenerate noise. Crisan and  McMurray \cite{7} studied the integration by parts formulae of solutions to SDEs with general McKean-Vlasov interactions. And theoretical research on DDSDE, including traversal \cite{4}, Harnack inequality \cite{5}, Bismut formula \cite{6}, etc.
	
	Similar to ordinary SDEs, it is also difficult to obtain the true solutions of DDSDEs. Therefore, developing numerical methods adapted to DDSDEs becomes particularly crucial. In practice, Monte Carlo ideas, stochastic particle methods, and the theory of propagation of chaos are commonly employed to construct numerical methods for such equations. Bossy and Talay \cite{8} is one of the earliest works to study the numerical approximation of DDSDE and establishes convergence of an Euler-Maruyama particle scheme with strong order
	$1/2$ in both the time-step and the number of particles under Lipschitz conditions. Reis, Engelhardt and Smith \cite{15} drew inspiration from the design ideas in \cite{23,36} and proposed a tamed Euler scheme, study of stable time-stepping schemes for interacting particle systems with a drift that grows super-linearly in the state component and with globally Lipschitz continuous diffusion term.
	He, Gao, Zhan and Guo \cite{JSW} propose an Euler?Maruyama
	method for McKean?Vlasov SDEs driven by fractional Brownian motion.
	
	Fractional Brownian motion is a generalization of standard Brownian motion, which is a non-Markov Gaussian process. The fBm exhibits long-range dependence for $H> 1/2$,
	and short-range dependence for $H < 1/2$. FBm with long-term memory characteristics can compensate for the shortcomings of standard Brownian motion in solving memory problems and construct more realistic market models. Such as, Modeling Complex Phenomena(see, for example, Baillie \cite{RTB}, Sottinen \cite{TPS}), Applications in Physics(see, for example, Shen and Tsai \cite{SWS}), Advanced Mathematical Research(we can see  Biagini,  Hu, {\O}ksendal and  Zhang \cite{Biagini}).
	Note that, for the   McKean-Vlasov SDEs driven by additive fBm, Fan et al. \cite{FH}   showed the well-posedness and derived a Bismut type formula for the Lions derivative using Malliavin calculus. Galeati et al. \cite{Galeati} studied McKean-Vlasov SDEs with irregular, possibly distributional drift, driven by additive fBm of Hurst parameter $H\in (0,1)$ and established strong well-posedness under a variety of assumptions on the drifts. Shen, Xiang and Wu \cite{xj} studied averaging principle of McKean-Vlasov SDEs driven simultaneously by fBm with Hurst index $H>\frac12$ and standard Brownian motion under certain averaging conditions.
	Bauer and Meyer-Brandis \cite{Bauer} established existence and uniqueness results of solutions to McKean-Vlasov equations driven by cylindrical fBm in an infinite-dimensional Hilbert space setting with irregular drift. Buckdahn and Jing \cite{Buckdahn1} considered mean-field SDEs driven by fBm and related stochastic control problems.
	Recently, Fan and Zhang \cite{fan} introduced a H\"older space of probability measure paths
	which is a complete metric space under a new metric, this solves the existence and uniqueness of DDSDE driven by multiplicative fBm, as well as large and moderate deviation principles.
	
	As far as we know, for DDSDE, some literatures have studied the case driven by additive fBM, but there are few studies on multiplicative noise. In this paper, consider DDSDE driven by   multiplicative fBm of the form
	\begin{equation}\label{eq-1}
		\begin{aligned}
			\d X_t&=b(X_t,\L_{X_t})\d t+\sigma(X_t,\L_{X_t})\d B^H_t,\quad t\in [S,T],
		\end{aligned}
	\end{equation}
	with the initial value $X_S$ is an $\F_S$-measurable random variable, $S\in[0,T)$. $b:\R^d  \times \pb_\theta(\R^d) \to \R^d$, $\sigma:\R^d \times \pb_\theta(\R^d) \to \R^{d\times d}$ are Borel measurable functions? and $B^H_t$ is a $d$-dimensional fBm with Hurst parameter $H \in (\frac{1}{2}, 1)$.
	
	The structure of this work is as follows. In Section 2, we provide the necessary notations and preliminary knowledge. The study on the propagation of chaos is presented in Section 3. In Section 4, we prove the Euler-Maruyama method for Euler-Maruyama method of DDSDE driven by multiplicative fractional Brownian motion. Numerical simulations are presented in Section 5.
	
	Throughout the paper, we let $C$ stand for positive constants, with or without
	subscript, and their value may be different in different appearances.

	\section{Preliminaries}
	\subsection{Notation.} In this paper, we need to use the following notations.
	We use $|\cdot|$ and $?\cdot, \cdot?$ for the Euclidean norm and inner product, respectively. Let $a \wedge b = \min\{a, b\}$ and $a \vee b = \max\{a, b\}$.
	$C([a,b];\R^d)$ denotes the space of all $\R^d$-valued continuous functions $\varphi:[a,b]\ra\R^d$ with the norm
	$\|\varphi\|_{a,b,\infty}=\sup_{a\leq s\leq b}|\varphi(s)|.$
	For any $\alpha\in(0,1), C^\alpha([a,b];\R^d)$ is the set of all $\R^d$-valued $\alpha$-H\"{o}lder continuous functions.
	If $\varphi\in C^\alpha([a,b];\R^d)$, we will make use of the notation
	\begin{align}\label{In8}
		\|\varphi\|_{a,b,\alpha}=\sup\limits_{a\leq s<t\leq b}\frac{|\varphi(t)-\varphi(s)|}{|t-s|^\alpha}.
	\end{align}
	When the interval is $[0,T]$, we abbreviate $\|\varphi\|_{0,T,\infty}$ and $\|\varphi\|_{0,T,\alpha}$ as $\|\varphi\|_{T,\infty}$ and $\|\varphi\|_{T,\alpha}$, respectively.
	Let $$\mathscr{P}_{\theta}(\mathbb{R}^{d}):=\Big\{\mu\in{\mathscr{P}}(\mathbb{R}^{d}):\mu(|\cdot|^{\theta}):=
	\int_{\mathbb{R}^{d}}|x|^{\theta}\mu(dx)<\infty\Big\},\quad \theta\in[1,\infty),$$
	where ${\mathscr{P}}$ is the set of probability measure on $(\mathbb{R}^{d},\B(\mathbb{R}^{d}))$. The space $\mathscr{P}_{\theta}(\mathbb{R}^{d})$ is a Polish space under the $L^{\theta}$-Wasserstein distance
	$${W}_{\theta}(\mu_{1},\mu_{2}):=\inf_{\pi\in\mathscr{C}(\mu_{1},\mu_{2})}\Big(\int_{\mathbb{R}^{d}
		\times\mathbb{R}^{d}}|x-y|^{\theta}\pi(dx,dy)\Big)^{\frac{1}{\theta}},\quad \mu_{1},\mu_{2}\in\mathscr{P}_{\theta}(\mathbb{R}^{d}),$$
	where $\mathscr{C}(\mu_{1},\mu_{2})$ is the set of probability measures on $\mathbb{R}^{d}\times\mathbb{R}^{d}$ with marginals $\mu_{1}$ and $\mu_{2}$, respectively.
	
	\subsection{Fractional integral and derivative}
	
	Let $a,b\in\R$ with $a<b$.
	For $f\in L^1([a,b],\R)$ and $\alpha>0$, the left-sided (respectively right-sided) fractional Riemann-Liouville integral of $f$ of order $\alpha$
	on $[a,b]$ is defined as
	\begin{align}\label{FrIn}
		&I_{a+}^\alpha f(x)=\frac{1}{\Gamma(\alpha)}\int_a^x\frac{f(y)}{(x-y)^{1-\alpha}}\d y\\
		&\qquad\left(\mbox{respectively}\ \ I_{b-}^\alpha f(x)=\frac{(-1)^{-\alpha}}{\Gamma(\alpha)}\int_x^b\frac{f(y)}{(y-x)^{1-\alpha}}\d y\right),\nonumber
	\end{align}
	where $x\in(a,b)$ a.e., $(-1)^{-\alpha}=\e^{-i\alpha\pi}$ and $\Gamma$ denotes the Gamma function.
	In particular, if $\alpha=n\in\N$, they are consistent with the usual $n$-order iterated integrals.
	
	Fractional differentiation can be given as an inverse operation.
	Let $\alpha\in(0,1)$ and $p\geq1$.
	If $f\in I_{a+}^\alpha(L^p([a,b],\R))$ (respectively $I_{b-}^\alpha(L^p([a,b],\R)))$, then the function $g$ satisfying $I_{a+}^\alpha g=f$ (respectively $I_{b-}^\alpha g=f$) is unique in $L^p([a,b],\R)$ and it coincides with the left-sided (respectively right-sided) Riemann-Liouville derivative
	of $f$ of order $\alpha$ shown by
	\begin{align*}
		&D_{a+}^\alpha f(x)=\frac{1}{\Gamma(1-\alpha)}\frac{\d}{\d x}\int_a^x\frac{f(y)}{(x-y)^\alpha}\d y\\
		&\qquad\left(\mbox{respectively}\ D_{b-}^\alpha f(x)=\frac{(-1)^{1+\alpha}}{\Gamma(1-\alpha)}\frac{\d}{\d x}\int_x^b\frac{f(y)}{(y-x)^\alpha}\d y\right).
	\end{align*}
	The corresponding Weyl representation is of the form
	\begin{align}\label{FrDe}
		&D_{a+}^\alpha f(x)=\frac{1}{\Gamma(1-\alpha)}\left(\frac{f(x)}{(x-a)^\alpha}+\alpha\int_a^x\frac{f(x)-f(y)}{(x-y)^{\alpha+1}}\d y\right)\\
		&\qquad\left(\mbox{respectively}\ \ D_{b-}^\alpha f(x)=\frac{(-1)^\alpha}{\Gamma(1-\alpha)}\left(\frac{f(x)}{(b-x)^\alpha}+\alpha\int_x^b\frac{f(x)-f(y)}{(y-x)^{\alpha+1}}\d y\right)\right),\nonumber
	\end{align}
	where the convergence of the integrals at the singularity $y=x$ holds pointwise for almost all $x$ if $p=1$ and in the $L^p$ sense if $p>1$.
	For more details, we refer the reader to \cite{SKM93}.
	
	Suppose that $f\in C^\lambda([a,b];\R^d)$ and $g\in C^\mu([a,b];\R^d)$ with $\lambda+\mu>1$.
	By \cite{Young36a}, the Riemann-Stieltjes integral $\int_a^bf\d g$ exists.
	In \cite{Zahle}, Z\"{a}hle provides an explicit expression for the integral $\int_a^bf\d g$ in terms of fractional derivatives.
	That is, let $\lambda>\alpha$ and $\mu>1-\alpha$ with $\alpha\in(0,1)$.
	Then the Riemann-Stieltjes integral can be expressed as
	\begin{align}\label{fb}
		\int_a^bf\d g=(-1)^\alpha\int_a^bD_{a+}^\alpha f(t)D_{b-}^{1-\alpha}g_{b-}(t)\d t,
	\end{align}
	where $g_{b-}(t)=g(t)-g(b)$.
	The relation \eqref{fb} can be regarded as fractional integration by parts formula.
	
	\subsection{fractional Brownian motion}
	Recall that the fBm $B^{H}=(B^{H,1},\cdots,B^{H,d})$  with Hurst index $H\in
	(0,1)$  is a centered Gaussian process, whose covariance structure is defined by
	$$\mathbb{E}(B_{t}^{H,i}B_{s}^{H,j})=R_{H}(t,s)\delta_{i,j},\quad s,t\in[0,T],\quad i,j=1,\cdots,d$$
	with $R_{H}(t,s)=\frac{1}{2}(t^{2H}+s^{2H}-|t-s|^{2H})$.
	Thus, Kolmogorov's continuity criterion implies that fBm is H\"{o}lder continuous of order $\delta$ for any $\delta<H.$
	Besides, $R_{H} (t, s)$ has the following integral representation
	$$R_{H}(t,s)=\int_{0}^{t\wedge s}K_{H}(t,r)K_{H}(s,r)dr,$$
	where the deterministic kernel $K_{H}(t,s)$ is   given by
	$$K_{H}(t,s)=C_{H}s^{\frac{1}{2}-H}\int_{s}^{t}(u-s)^{H-\frac{3}{2}}u^{H-\frac{1}{2}}du,\quad t>s,$$
	where $C_{H}=\sqrt{\frac{H(2H-1)}{\mathcal{B}(2-2H,H-1/2)}}$ and $\mathcal{B}$ standing for the Beta function. If $t\leq s$, we set $K_{H}(t,s)=0$.
	For
	$H=\frac1 2$, $B^H$ coincides with the standard Brownian motion $B$, but
	$B^H$ is neither a semimartingale nor a Markov process unless
	$H=\frac 12$. As a consequence, some classical techniques of stochastic
	analysis are not applicable.
	
	\subsection{The Lions derivative}
	Let $\|\cdot\|_{L^2_\mu}$ denotes for the norm of the space $ L^2(\R^d\ra\R^d,\mu)$ and for a random variable $X$, $\L_X$ denotes its distribution.
	
	\begin{definition}
		Let $f:\mathscr{P}_2(\R^d)\ra\R$ and $g:\R^d\times\mathscr{P}_2(\R^d)\ra\R$.
		\begin{enumerate}
			\item[(1)]  $f$ is called $L$-differentiable at $\mu\in\mathscr{P}_2(\R^d)$, if the functional
			\begin{align*}
				L^2(\R^d\ra\R^d,\mu)\ni\phi\mapsto f(\mu\circ(\mathrm{Id}+\phi)^{-1}))
			\end{align*}
			is Fr\'{e}chet differentiable at $0\in L^2(\R^d\ra\R^d,\mu)$. That is, there exists a unique $\gamma\in L^2(\R^d\ra\R^d,\mu)$ such that
			\begin{align*}
				\lim_{\|\phi\|_{L^2_\mu}\ra0}\frac{f(\mu\circ(\mathrm{Id}+\phi)^{-1})-f(\mu)-\mu(\langle\gamma,\phi\rangle)}{\|\phi\|_{L^2_\mu}}=0.
			\end{align*}
			In this case, $\gamma$ is called the $L$-derivative of $f$ at $\mu$ and denoted by $D^Lf(\mu)$.
			
			\item[(2)] $f$ is called $L$-differentiable on $\mathscr{P}_2(\R^d)$, if the $L$-derivative $D^Lf(\mu)$ exists for all $\mu\in\mathscr{P}_2(\R^d)$.
			Furthermore, if for every $\mu\in\mathscr{P}_2(\R^d)$ there exists a $\mu$-version $D^Lf(\mu)(\cdot)$ such that $D^Lf(\mu)(x)$ is jointly continuous in $(\mu,x)\in\mathscr{P}_2(\R^d)\times\R^d$, we denote $f\in C^{(1,0)}(\mathscr{P}_2(\R^d))$.
			
			\item[(3)] $g$ is called differentiable on $\R^d\times\mathscr{P}_2(\R^d)$, if for any $(x,\mu)\in\R^d\times\mathscr{P}_2(\R^d)$,
			$g(\cdot,\mu)$ is differentiable and $g(x, \cdot)$ is $L$-differentiable.
			If $\nabla g(\cdot,\mu)(x)$ and $D^Lg(x,\cdot)(\mu)(y)$ are jointly continuous in $(x,y,\mu)\in\R^d\times\R^d\times\mathscr{P}_2(\R^d)$,
			we denote $g\in C^{1,(1,0)}(\R^d\times\mathscr{P}_2(\R^d))$.
			If moreover the derivatives
			$$\nabla^2g(\cdot,\mu)(x),\  \nabla(D^Lg(\cdot,\cdot)(\mu)(y))(x),\ \nabla(D^Lg(x,\cdot)(\mu)(\cdot))(y),\ D^L(D^Lg(x,\cdot)(\cdot)(y))(\mu)(z),$$
			exist and are jointly continuous in the corresponding arguments $(x,\mu), (x,\mu,y)$ or $(x,\mu,y,z)$,
			we denote $g\in C^{2,(2,0)}(\R^d\times\mathscr{P}_2(\R^d))$.
			If $g\in C^{2,(2,0)}(\R^d\times\mathscr{P}_2(\R^d))$ and with all these derivatives is bounded on $\R^d\times\mathscr{P}_2(\R^d)$,
			we denote $g\in C_b^{2,(2,0)}(\R^d\times\mathscr{P}_2(\R^d))$.
		\end{enumerate}
	\end{definition}
	In order to ease notations, we denote $D^{L,2}g=D^L(D^Lg)$, and for a vector-valued function $f=(f_i)$ or a matrix-valued function $f=(f_{ij})$ with $L$-differentiable components, we simply write
	\begin{align*}
		D^Lf(\mu)=(D^Lf_i(\mu))  \ \  \mathrm{or}\ \ D^Lf(\mu)=(D^Lf_{ij}(\mu)).
	\end{align*}
	Let us finish this part by giving a useful formula for the $L$-derivative that are needed later on, which is due to \cite[Theorem 6.5]{Cardaliaguet13} and \cite[Proposition 3.1]{6}.
	\begin{lemma}\label{FoLD}
		Let $(\Omega,\mathscr{F},\mathbb{P})$ be an atomless probability space and $X,Y\in L^2(\Omega\ra\R^d,\mathbb{P})$.
		If $f\in C^{(1,0)}(\mathscr{P}_2(\R^d))$, then
		\begin{align*}
			\lim_{\epsilon\da0}\frac {f(\L_{X+\epsilon Y})-f(\L_X)} \epsilon=\E\langle D^Lf(\L_X)(X),Y\rangle.
		\end{align*}
	\end{lemma}
	
	\subsection{H\"older space of probability measure paths}
	When addressing the uniqueness of solutions to the equation \eqref{eq-1}, it is necessary to consider the integral between the following two solutions $X_t,Y_t$:
	\begin{equation}\label{int-si-si}
		\begin{split}
		&\int_s^t(\sigma(X_r,\L_{X_r})-\sigma(Y_r,\L_{Y_r}))\d B^H_r\\
		&=(-1)^{\alpha}\int_s^tD_{s+}^\alpha(\sigma(X_\cdot,\L_{X_\cdot})-\sigma(Y_\cdot,\L_{Y_\cdot}))(r)D_{t-}^{1-\alpha}B_{t-}^H(r)\d r\\
		&=\frac{(-1)^{\alpha}}{\Gamma(1-\alpha)}\int_s^t
		\bigg[\bigg(\frac{\sigma(X_r ,\L_{X_r})-\sigma(Y_r,\L_{Y_r})}{(r - s)^\alpha}\cr
		&\quad\quad+\alpha\int_s^r\frac{\sigma(X_r,\L_{X_r})-\sigma(Y_r,\L_{Y_r})-\left(\sigma(X_u,\L_{X_u})-\sigma(Y_u,\L_{Y_u})\right)}
		{(r - u)^{1+\alpha}}\d u\bigg)\bigg]D_{t-}^{1-\alpha}B_{t-}^H(r)\d r.
	\end{split}
	\end{equation}
	Since $C^\beta([0,T];\R ^d)$ is not a separable space under the H\"older norm, i.e. $C^\beta([0,T];\R ^d)$ is not a Polish space under the H\"older norm,  it is essential to introduce $\W_{2,T,\beta}(\mu,\nu)$ to control the distribution part in the stochastic integral \eqref{int-si-si}. Let $T>0$ be fixed in this part. To simplify the notation, we set $W_T^d:=C([0,T];\R^d)$, and let $\mathscr P_2(W_T^d)$ be all probability measures $\mu$ on $W_T^d$ such that $\mu(\|\cdot\|_{T,\infty}^2)<+\infty$,
	in which the distance is defined as
	\begin{align}\label{AddHO1}
		\W_{2,T}(\mu,\nu)=\inf_{\pi\in \mathscr C(\mu,\nu)}\left(\int_{W_T^d\times W_T^d}\|\gamma_1-\gamma_2\|_{T,\infty}^2\pi(\d \gamma_1,\d \gamma_2)\right)^\frac 1 2.
	\end{align}
	For $\mu\in\mathscr P_2(W_T^d)$ and $0\leq s<t\leq T$, let $\mu_s$ and $\mu_{s,t}$ stand for the respective marginals of $\mu$ at $s$ and $(s,t)$,
	and let $\mu_{s,t}^\Delta$ be the distribution of the following random variable on $(\R^{2d},\mathscr B(\R^{2d}),\mu_{s,t})$:
	\[\R^d\times\R^d\ni(x,y)\mapsto x-y\in\R^d.\]
	Now, for fixed $\beta\in(0,1)$, we define
	\begin{align}\label{AddHO3}
		\mathscr P_{2,\beta}(W_T^d):=\left\{\mu\in \mathscr P_2(W_T^d)~\Big|~\sup_{0\leq s<t\leq T}\frac {\sqrt{\mu_{s,t}^\Delta (|\cdot|^2)}} {(t-s)^\beta}<+\infty\right\}.
	\end{align}
	For any $x=(x_1,x_2),y=(y_1,y_2)\in\R^{2d}$, we first let
	\begin{align*}
		|x-y|_M=|x_1-y_1|\vee |x_2-y_2|,
	\end{align*}
	and for any $\mu,\nu\in\mathscr P_2(W_T^d)$ and any $0\leq s_1<s_2\leq T$, define
	\begin{align*}
		\W_{2}(\mu_{s_1,s_2},\nu_{s_1,s_2}):=\inf_{\pi_{s_1,s_2}\in\mathscr C(\mu_{s_1,s_2},\nu_{s_1,s_2})}\left(\int_{\R^{2d}\times\R^{2d}} |x-y|_M^2\pi_{s_1,s_2}(\d x ,\d y)\right)^\frac 1 2.
	\end{align*}

	Let $\mu,\nu\in\mathscr P_2(W_T^d)$. For $0\leq s_1<s_2\leq T$, and a coupling $\pi_{s_1,s_2}\in \mathscr C(\mu_{s_1,s_2},\nu_{s_1,s_2})$, let $\pi_{s_1}$ and $\pi_{s_2}$  be the marginal distributions of $\pi_{s_1,s_2}$ at $s_1$ and $s_2$ respectively. Then $\pi_{s_i}\in\mathscr C(\mu_{s_i},\nu_{s_i})$, $i=1,2$. Denote by $\mathscr C_{opt}(\mu_{s_1,s_2},\nu_{s_1,s_2})$ the optimal couplings of $(\mu_{s_1,s_2},\nu_{s_1,s_2})$ with respect to the above $\W_2$-distance, 
	and we define for any $\mu,\nu\in\mathscr P_2(W_T^d)$,
	\begin{align}\label{AddHO2}
		W_{2,s_1,s_2}^c(\mu,\nu)=\inf\left\{\mathscr q{\pi_{s_1,s_2}(c)}~\Big|~\pi_{s_1,s_2}\in \mathscr C_{opt}(\mu_{s_1,s_2},\nu_{s_1,s_2})\right\},
	\end{align}
	where
	$\pi_{s_1,s_2}(c):=\int_{\R^{2d}\times\R^{2d}} c(x,y)\pi_{s_1,s_2}(\d x ,\d y)$ with $c$ giving by
	\[c(x,y)=|x_1-y_1-(x_2-y_2)|^2, \ \ x=(x_1,x_2),y=(y_1,y_2)\in\R^{2d}.\]	
	Let
	\begin{equation}\label{3.10}
		\W_{2,T,\beta}(\mu,\nu)=\W_{2,T}(\mu,\nu)+\sup_{0\leq s_1<s_2\leq T}\frac {W_{2,s_1,s_2}^c(\mu,\nu)} {(s_2-s_1)^\beta},\ \ \mu,\nu\in\mathscr P_{2,\beta}(W_T^d).
	\end{equation}
	We can learn that The space $\mathscr P_{2,\beta}(W_T^d)$ is a complete metric space under the metric $\W_{2,T,\beta}$ (For more details, see Theorem 3.4 in \cite{fan}). Noting that the above results can extend naturally to the subinterval of $[0,T]$.
	That is, for each $S\in[0,T)$, let $W_{S,T}^d=C([S,T];\R^d)$, and $\W_{2,S,T}, \W_{2,S,T,\beta}$ be defined on $\mathscr P_{2,\beta}(W_{S,T}^d)$ similarly
	to \eqref{AddHO1} and \eqref{3.10}, then $(\mathscr P_{2,\beta}(W_{S,T}^d),\W_{2,S,T,\beta})$ is a complete metric space.
	Finally,let
	\begin{align}\label{AddDef}
		\|\mu\|_{2,S,T,\beta}&=\sqrt{\mu(\|\cdot\|_{S,T,\infty}^2)}+\sup_{S\leq s<t\leq T}\frac {\sqrt{\mu_{s,t}^\Delta(|\cdot|^2)}} {(t-s)^\beta},
	\end{align}
	and we simply write $\|\mu\|_{2,T,\beta}$ as $\|\mu\|_{2,S,T,\beta}$ if $S=0$. We remark that $\|\mu\|_{2,S,T,\beta}=\W_{2,S,T,\beta}(\mu,\de_{\bf 0})$.

	\subsection{Well-posedness}
	\begin{assumption}\label{ass-2.3}
		There exists a constant $K > 0$ such that
		\begin{align*}
			|b(x,\mu)-b(y,\nu)|\leq K(|x-y|+W_2(\mu,\nu)),\quad x,y\in \R^d, \mu,\nu\in\pb_2(\R^d),
		\end{align*}
		and $\sigma_{ij}\in C_b^{2,(2,0)}(\R^d\times\pb_2(\R^d)),1\leq i,j \leq d$.
	\end{assumption}
	
	\begin{definition}
		We say that $\{X_t\}_{t\in[0,T]}$ is a solution of equation \eqref{eq-1}, if $\{X_t\}_{t\in[0,T]}$ is an adapted process such that for any $\beta\in(0,H)$, $\mathbb P$-a.s. $X\in C^{\beta}([0,T];\R^d)$,
		\begin{equation}\label{in-mom}
			\E\left(\|X\|_{T,\infty}^2+\|X\|_{T,\beta}^2\right)<+\infty,
		\end{equation}
		and $X_t$ satisfies equation \eqref{eq-1}.
	\end{definition}
	
	\begin{lemma}\label{lem-2.5}\cite[Lemma 3.7]{fan}
		Let $\beta\in (\frac 1 2,H)$, $\beta_1\in [\beta,H)$ and $1-\beta<\alpha<\beta$.
		Assume that  $\mu,\nu\in\mathscr P_{2,\beta}(W_{S,T}^d)$ and $\sigma_{ij}\in  C_b^{2,(2,0)}(\R^d\times\mathscr P_2(\R^d)), 1\leq i,j\leq d$.
		Then for any $S\leq s<t\leq T$ with $t-s\leq 1$,
		\begin{align*}
			&\left|\int_s^t\left(\sigma(X_r,\mu_r)-\sigma(Y_r,\nu_r)\right)\d B^H_r\right|\cr
			\le&\|B^H\|_{s,t,\beta_1} \left(\Lambda_1 (t-s)^{\beta_1}+\Lambda_2  (t-s)^{\alpha+\beta_1}\right)\|X-Y\|_{s,t,\infty}\cr
			&+\Lambda_3 \|B^H\|_{s,t,\beta_1} (t-s)^{\beta+\beta_1}\|X-Y\|_{s,t,\beta}\cr
			&+\|B^H\|_{s,t,\beta_1} \left(\Lambda_4 (t-s)^{\beta_1}+\Lambda_5 (t-s)^{\alpha+\beta_1}\right)\W_{2,S,T,\beta}(\mu ,\nu),
		\end{align*}
		where
		\begin{equation}\label{1-EsNoi}
			\begin{split}
			&\Lambda_1 := \frac {C_0\mathcal{B}(1-\alpha,\alpha+\beta_1)} {\Gamma(1-\alpha)} \|\nabla\sigma\|_\infty, \\
			&\Lambda_2 := \frac {2^{3-\frac {\alpha} {\beta}} \beta C_0\|\nabla\sigma\|_\infty^{\frac {\beta-\alpha} {\beta}}} {(\beta-\alpha)(\alpha+\beta_1)\Gamma(1-\alpha)}\Big[\left( \|\nabla^2\sigma\|_\infty \left( {\|X\|_{s,t,\beta}\vee\|Y\|_{s,t,\beta}} \right)\right)^{\frac {\alpha} {\beta}}\cr
			&\qquad +\left( \|D^L\nabla\sigma\|_\infty \left(\W_{2,S,T,\beta}(\mu,\de_{\bf 0})\wedge \W_{2,S,T,\beta}(\nu,\de_{\bf 0})\right)\right)^{\frac {\alpha} {\beta}} \Big], \\
			&\Lambda_3 :=\frac {\alpha C_0\|\nabla\sigma\|_\infty\mathcal{B}(\beta-\alpha+1,\alpha+\beta_1)} {(\beta-\alpha)\Gamma(1-\alpha)}, \\
			&\Lambda_4 :=\frac {C_0\|D^L \sigma\|_\infty } {\Gamma(1-\alpha)} \left( \mathcal{B}(1-\alpha,\alpha+\beta_1)+\frac {\alpha  \mathcal{B}(\beta-\alpha+1,\beta_1+\alpha)} {\beta-\alpha} \right), \cr
			&\Lambda_5 :=\frac {2^{3-\frac {\alpha} {\beta}} \beta C_0 \|D^L\sigma\|_\infty^{\frac {\beta-\alpha} {\beta}}} {(\beta-\alpha)(\beta_1+\alpha)\Gamma(1-\alpha)} \Big[ (\|\nabla_1 D^L\sigma\|_\infty (\|X\|_{s,t,\beta}\wedge\|Y\|_{s,t,\beta}))^{\frac {\alpha} {\beta}}\cr
			&\qquad + \left((\|D^{L,2}\sigma\|_\infty+\|\nabla_2 D^L\sigma\|_\infty)( \W_{2,S,T,\beta}(\mu,\de_{\bf 0})\vee\W_{2,S,T,\beta}(\nu,\de_{\bf 0})) \right)^{\frac {\alpha}{\beta}}  \Big],
		\end{split}
	\end{equation}
	in which $C_0:=\frac{\beta_1}{\Gamma(\alpha)(\alpha+\beta_1 -1)}$, $\nabla_1 D^L\sigma$ and $\nabla_2 D^L\sigma$ stand for the gradient operators of $D^L\sigma(x,\cdot)(\mu)(y)$ on $x$ and $y$, respectively.
	\end{lemma}
	
	\begin{theorem}
		Let $H\in (\frac{\sqrt{5}-1}{2},1)$ and $T > 0$. Assume that Assumption \ref{ass-2.3} holds and $\E e^{{|X_0|}^{\frac{2(1-H)}{H}+\epsilon_0}}< +\infty$
		some constant $\epsilon_0 > 0$. According to Theorem 3.1 in \cite{fan}, we get equation \eqref{eq-1} is well-posedness on $[0, T]$.
	\end{theorem}
	
	\section{Propagation of chaos}
	In this section, we consider the  interacting particle system
	\begin{equation}\label{eq-4.1}
		X^{i,N}_t=X_S^i+\int_{S}^{t}b(X^{i,N}_s,\mu_s^N)\d s+\int_{S}^{t}\sigma(X^{i,N}_s,\mu_s^N)\d B^{H,i}_s, \quad S\leq t\leq T,
	\end{equation}
	where
	$\mu_t^N=\frac 1N\sum_{i=1}^N\delta_{X_t^{i,N}}$, almost surely for any $S \in [0,T )$ and $i \in {1,...,N}$. The aim of this section is to study the  limit behavior of the system \eqref{eq-4.1} as $N\rightarrow\infty.$  To this end, we introduce the following system of noninteracting particles:
	\begin{equation}\label{eq-4.2}
		X^i_t=X_S^i+\int_{S}^{t}b(X^i_s,\L_{X^i_s})\d s+\int_{S}^{t}\sigma(X^i_s,\L_{X^i_s})\d B^{H,i}_s,\quad S\leq t\leq T.
	\end{equation}
	
	\begin{lemma}\label{lem3.1}
		For two empirical measure $\mu_{t}^{N}=\frac{1}{N}\sum_{j=1}^{N}\delta_{X_t^{j,N}}$ and $\nu_{t}^{N}=\frac{1}{N}\sum_{j=1}^{N}\delta_{{X}^j_{t}}.$ Then,
		\begin{equation}
			\W_{2,S,T,\beta}(\mu^{N},\nu^{N})\leq \sqrt{\frac{1}{N}\sum_{j=1}^N\|X^j-X^{j,N}\|^2_{S,T,\infty}}+\sqrt{\frac{1}{N}\sum_{j=1}^N\|X^j-X^{j,N}\|^2_{S,T,\beta}}.
		\end{equation}	
	\end{lemma}
	\begin{proof}
		 By substituting coupling $\pi_1=\frac{1}{N}\sum_{j=1}^N\delta_{(X^j,X^{j,N})}\in \mathscr C (\mu^N,\nu^N)$ into \eqref{AddHO1}, we obtain
		 \begin{align*}
		 	\W_{2,S,T}(\mu^N,\nu^N)&\leq \Big(\int_{W^d_{S,T} \times W^d_{S,T}}\|\gamma_1-\gamma_2\|^2_{S,T,\infty}\frac{1}{N}\sum_{j=1}^N\delta_{(X^j,X^{j,N})}(\d \gamma_1,\d \gamma_2)\Big)^{\frac{1}{2}}\\
		 	&\leq \Big(\frac{1}{N}\sum_{j=1}^N\int_{W^d_{S,T} \times W^d_{S,T}}\|\gamma_1-\gamma_2\|^2_{S,T,\infty}\delta_{(X^j,X^{j,N})}(\d \gamma_1,\d \gamma_2)\Big)^{\frac{1}{2}}\\
		 	&\leq \Big(\frac{1}{N}\sum_{j=1}^N\|X^j-X^{j,N}\|^2_{S,T,\infty}\Big)^{\frac{1}{2}}.
		 \end{align*}
		 Similarly, choosing the coupling $\pi_2=\frac{1}{N}\sum_{j=1}^N\delta_{((X^j_{s_1},X^j_{s_2}),(X^{j,N}_{s_1},X^{j,N}_{s_2}))}\in \mathscr{C}_{opt}(\mu^N_{s_1,s_2},\nu^N_{s_1,s_2})$ and substituting it into \eqref{AddHO2} yields that
		 \begin{align*}
		 	W_{2,s_1,s_2}^c(\mu^N,\nu^N)&\leq \Big( \int_{\R^{2d}\times \R^{2d}}|x_1-y_1-(x_2-y_2)|^2\frac{1}{N}\sum_{j=1}^N\delta_{((X^j_{s_1},X^j_{s_2}),(X^{j,N}_{s_1},X^{j,N}_{s_2}))}(\d x, \d y)\Big)^{\frac{1}{2}}\\
		 	&\leq \Big( \frac{1}{N}\sum_{j=1}^N |X^j_{s_1}-X^{j,N}_{s_1}-(X^j_{s_2}-X^{j,N}_{s_2})|^2\Big)^{\frac{1}{2}}\\
		 	&\leq \Big( \frac{1}{N}\sum_{j=1}^N \|X^j-X^{j,N}\|^2_{S,T,\beta}\Big)^{\frac{1}{2}}(s_2-s_1)^{\beta}.
		 \end{align*}
		 It follows directly from the definition of \eqref{3.10} that
		 \begin{align*}
		 		\W_{2,T,\beta}(\mu^{N},\nu^{N})\leq \sqrt{\frac{1}{N}\sum_{j=1}^N\|X^j-X^{j,N}\|^2_{S,T,\infty}}+\sqrt{\frac{1}{N}\sum_{j=1}^N\|X^j-X^{j,N}\|^2_{S,T,\beta}}.
		 \end{align*}
	\end{proof}
	
	\begin{lemma}\label{lem-3.2}
		Let $\frac{1}{2}<\beta<H$, Assumption \ref{ass-2.3} hold, $X^{i,N}_S$ and $X^i_S$ are $\F_S$-measurable random variable, then for any $p\geq 1$,
		\begin{equation}
			\begin{aligned}
				&\Big[\E\Big(\|X^i\|^p_{S,T,\infty}+\|X^i\|^p_{S,T,\beta}+\|X^{i,N}\|^p_{S,T,\infty}+\|X^{i,N}\|^p_{S,T,\beta}\Big)\Big]< \infty,
			\end{aligned}
		\end{equation}
		where $S\in [0,T)$.
	\end{lemma}
	\begin{proof}
		For any $S\leq s < t \leq T$,
		\begin{equation}
			X_t^{i,N}-X_s^{i,N}=\int_{s}^{t}b(X_r^{i,N},\mu_r^N)\d r+\int_{s}^{t}\sigma(X_r^{i,N},\mu_r^N)\d B^{H,i}_r.
		\end{equation}
		Applying the fractional integration by parts formula \eqref{fb} with $\alpha\in(1-\beta,\beta)$, we have
		\begin{equation}\label{eq4.6}
			\int_{s}^{t}\sigma(X_r^{i,N},\mu_r^N)\d B^{H,i}_r=(-1)^{\alpha}\int_{s}^{t}D^{\alpha}_{s+}\sigma(X_{.}^{i,N},\mu_{.}^N)(r)D^{1-\alpha}_{t-}B_{t-}^{H,i}(r)\d r.
		\end{equation}
		Based on Assumption \ref{ass-2.3} and lemma 5.1 in \cite{fan}, we obtain
		\begin{equation}\label{eq4.7}
			\begin{aligned}
				&|D^{\alpha}_{s+}\sigma(X_{.}^{i,N},\mu_{.}^N)(r)|\\
				&= \frac{1}{\Gamma(1-\alpha)}\Big|\frac{\sigma(X^{i,N}_r,\mu_r^N)}{(r-s)^{\alpha}}+\alpha\int_{s}^{r}\frac{\sigma(X^{i,N}_r,\mu_r^N)-\sigma(X^{i,N}_u,\mu_u^N)}{(r-u)^{\alpha+1}}\d u\Big|\\
				&\leq \frac{1}{\Gamma(1-\alpha)}\Big|\frac{M}{(r-s)^{\alpha}}+\alpha\int_{s}^{r}
				\frac{\big(\|\nabla\sigma\|_{\infty}|X^{i,N}_r-X^{i,N}_u|+\|D^L\sigma\|_{\infty}W_2(\mu_r^N,\mu_u^N)\big)\wedge (2M)}{(r-u)^{\alpha+1}}\d u\big|\\
				&\leq \frac{1}{\Gamma(1-\alpha)}\Big|\frac{M}{(r-s)^{\alpha}}\\
				&\quad+\alpha\int_{s}^{r}
				\frac{\big(\|\nabla\sigma\|_{\infty}|X^{i,N}_r-X^{i,N}_u|+\|D^L\sigma\|_{\infty}(\frac{1}{N}\sum\limits_{j=1}^{N}|X^{j,N}_r-X^{j,N}_u|^2)^{\frac{1}{2}}\big)\wedge (2M)}{(r-u)^{\alpha+1}}\d u\Big|\\
				&\leq \frac{1}{\Gamma(1-\alpha)}\Big|\frac{M}{(r-s)^{\alpha}}\\
				&\quad+\alpha\int_{s}^{r}
				\frac{\big(\|\nabla\sigma\|_{\infty}\|X^{i,N}\|_{u,r,\beta}(r-u)^{\beta}+\|D^L\sigma\|_{\infty}(\frac{1}{N}\sum_{j=1}^N\|X^{j,N}\|^2_{u,r,\beta})^{\frac{1}{2}}(r-u)^{\beta}\big)\wedge (2M)}{(r-u)^{\alpha+1}}\d u\Big|\\
				&\leq \frac{1}{\Gamma(1-\alpha)}\Big|\frac{M}{(r-s)^{\alpha}}+\alpha\int_{s}^{r}
				\frac{\Big(\|\nabla\sigma\|_{\infty}\|X^{i,N}\|_{u,r,\beta}(r-u)^{\beta}\Big)\wedge (2M)}{(r-u)^{\alpha+1}}\d u\\
				&\quad+\alpha\int_{s}^{t}\frac{\big(\|D^L\sigma\|_{\infty}(\frac{1}{N}\sum_{j=1}^N\|X^{j,N}\|^2_{u,r,\beta})^{\frac{1}{2}}(r-u)^{\beta}\big)\wedge (2M)}{(r-u)^{\alpha+1}}\d u\Big|\\
				&\leq \frac{1}{\Gamma(1-\alpha)}\Big[M(r-s)^{-\alpha}+\frac{2^{3-\frac{\alpha}{\beta}}\beta}{\beta-\alpha}M^{\frac{\beta-\alpha}{\beta}}(\|\nabla\sigma\|_{\infty}\|X^{i,N}\|_{s,r,\beta})^{\frac{\alpha}{\beta}}\\
				&\quad+\frac{2^{3-\frac{\alpha}{\beta}}\beta}{\beta-\alpha}M^{\frac{\beta-\alpha}{\beta}}\bigg( \|D^L\sigma\|_{\infty}(\frac{1}{N}\sum_{j=1}^N\|X^{j,N}\|^2_{s,r,\beta})^{\frac{1}{2}}\bigg)^{\frac{\alpha}{\beta}}
				\Big].
			\end{aligned}
		\end{equation}
		In addition
		\begin{equation}\label{eq4.8}
			\begin{aligned}
				\left|D_{t-}^{1-\alpha} B_{t-}^{H,i}(r)\right| & =\frac{1}{\Gamma(\alpha)}\left|\frac{B_{r}^{H,i}-B_{t}^{H,i}}{(t-r)^{1-\alpha}}+(1-\alpha) \int_{r}^{t} \frac{B_{r}^{H,i}-B_{u}^{H,i}}{(u-r)^{2-\alpha}} \, \d u\right| \\
				& \leq \frac{\beta}{\Gamma(\alpha)(\alpha+\beta-1)}\left\| B^{H,i}\right\| _{S, T, \beta}(t-r)^{\alpha+\beta-1} .
			\end{aligned}
		\end{equation}
		Plugging \eqref{eq4.7} and \eqref{eq4.8} into \eqref{eq4.6} yields
		\begin{align*}
			&\Big|\int_{s}^{t}\sigma(X_r^{i,N},\mu_r^N)\d B^{H,i}_r\Big|\\
			&\leq C\|B^{H,i}\|_{S,T,\beta}\Big[ \int_{s}^{t}(r-s)^{\alpha}(t-r)^{\alpha+\beta-1}\d r\\
			&\quad\quad\quad+\int_{s}^{t}\Bigg((t-r)^{\alpha+\beta-1}\|X^{i,N}\|_{s,r,\beta}^{\frac{\alpha}{\beta}}+(t-r)^{\alpha+\beta-1}\bigg(\frac{1}{N}\sum_{j=1}^N\|X^{j,N}\|^2_{s,r,\beta}\bigg)^{\frac{\alpha}{2\beta}}\Bigg)\d r\Big]\\
			&\leq C\|B^{H,i}\|_{S,T,\beta}\Big[ (t-s)^{\beta}+(t-s)^{\alpha+\beta}\|X^{i,N}\|_{s,t,\beta}^{\frac{\alpha}{\beta}}+(t-s)^{\alpha+\beta}\bigg(\frac{1}{N}\sum_{j=1}^N\|X^{j,N}\|^2_{s,t,\beta}\bigg)^{\frac{\alpha}{2\beta}}  \Big].
		\end{align*}
		With respect to the drift term, due to the linear growth property of
		$b$, it is straightforward to observe that
		\begin{equation}
			\begin{aligned}
				\Big| \int_{s}^{t}b(X^{i,N}_r,\mu_r^N)\d r\Big|&\leq \int_{s}^{t} K(1+|X^{i,N}_r|+W_2(\mu_r^N,\delta_0))\d r\\
				&\leq K(t-s)+K\int_{s}^{t}|X^{i,N}_r|\d r+\int_{s}^{t}\bigg(\frac{1}{N}\sum_{j=1}^N|X^{j,N}_r|^2\bigg)^{\frac{1}{2}}\d r.
			\end{aligned}
		\end{equation}
		Therefore, we have
		\begin{equation}
			\begin{aligned}
				&|X^{i,N}_t-X^{i,N}_s|\\
				&\leq \Big| \int_{s}^{t}b(X^{i,N}_r,\mu_r^N)\d r\Big|+\Big|\int_{s}^{t}\sigma(X_r^{i,N},\mu_r^N)\d B^{H,i}_r\Big|\\
				&\leq K(t-s)+K\int_{s}^{t}|X^{i,N}_r|\d r+\int_{s}^{t}\bigg(\frac{1}{N}\sum_{j=1}^N|X^{j,N}_r|^2\bigg)^{\frac{1}{2}}\d r\\
				&\quad+C\|B^{H,i}\|_{S,T,\beta}\Big[ (t-s)^{\beta}+(t-s)^{\alpha+\beta}\|X^{i,N}\|_{s,t,\beta}^{\frac{\alpha}{\beta}}+(t-s)^{\alpha+\beta}\bigg(\frac{1}{N}\sum_{j=1}^N\|X^{j,N}\|^2_{s,t,\beta}\bigg)^{\frac{\alpha}{2\beta}}  \Big]\\
				&\leq K(t-s)+K(t-s)^{\beta}\Big(\int_{s}^{t}|X^{i,N}_r|^{\frac{1}{1-\beta}}\d r\Big)^{1-\beta}+K(t-s)^{\beta}\Bigg(\int_{s}^{t}\bigg(\frac{1}{N}\sum_{j=1}^N|X^{j,N}_r|^2\d r\bigg)^{\frac{1}{2(1-\beta)}}\Bigg)^{1-\beta}\\
				&\quad+C\|B^{H,i}\|_{S,T,\beta}\Big[ (t-s)^{\beta}+(t-s)^{\alpha+\beta}\|X^{i,N}\|_{s,t,\beta}^{\frac{\alpha}{\beta}}+(t-s)^{\alpha+\beta}\bigg(\frac{1}{N}\sum_{j=1}^N\|X^{j,N}\|^2_{s,t,\beta}\bigg)^{\frac{\alpha}{2\beta}}  \Big].
			\end{aligned}
		\end{equation}
		Consequently, it follows that
		\begin{equation}
			\begin{aligned}
				&\frac{|X^{i,N}_t-X^{i,N}_s|}{(t-s)^\beta}\\
				&\leq K(t-s)^{1-\beta}+K\Big(\int_{s}^{t}|X^{i,N}_r|^{\frac{1}{1-\beta}}\d r\Big)^{1-\beta}+K\Bigg(\int_{s}^{t}\bigg(\frac{1}{N}\sum_{j=1}^N|X^{j,N}_r|^2\d r\bigg)^{\frac{1}{2(1-\beta)}}\Bigg)^{1-\beta}\\
				&\quad+C\|B^{H,i}\|_{S,T,\beta}\Big[ 1+(t-s)^{\alpha}\|X^{i,N}\|_{s,t,\beta}^{\frac{\alpha}{\beta}}+(t-s)^{\alpha}\bigg(\frac{1}{N}\sum_{j=1}^N\|X^{j,N}\|^2_{s,t,\beta}\bigg)^{\frac{\alpha}{2\beta}}  \Big].
			\end{aligned}
		\end{equation}
		Young inequality point out that
		\begin{align*}
			(t-s)^{\alpha}\|X^{i,N}\|_{s,t,\beta}^{\frac{\alpha}{\beta}}\leq \frac{\beta-\alpha}{\beta}+\frac{\alpha}{\beta}(t-s)^{\beta}\|X^{i,N}\|_{s,t,\beta},
		\end{align*}
		and
		\begin{align*}
			(t-s)^{\alpha}(\frac{1}{N}\sum_{j=1}^N\|X^{j,N}\|^2_{s,t,\beta}\big)^{\frac{\alpha}{2\beta}}\leq \frac{\beta-\alpha}{\beta}+\frac{\alpha}{\beta}(t-s)^{\beta}\bigg(\frac{1}{N}\sum_{j=1}^N\|X^{j,N}\|^2_{s,t,\beta}\bigg)^{\frac{1}{2}}.
		\end{align*}
		Thus, we obtain that
		\begin{equation}\label{eq-4.13}
			\begin{aligned}
				&\|X^{i,N}\|_{s,t,\beta}\\
				&\leq K\big(1+|X^{i,N}_s|+\bigg(\frac{1}{N}\sum_{j=1}^N|X^{j,N}_s|^2\bigg)^{\frac{1}{2}}\big)(t-s)^{1-\beta}+K\bigg(\int_{s}^{t}|X^{i,N}_r-X^{i,N}_s|^{\frac{1}{1-\beta}}\d r\bigg)^{1-\beta}\\
				&\quad+K\Big(\int_{s}^{t}(\frac{1}{N}\sum_{j=1}^N|X^{j,N}_r-X^{i,N}_s|^2)^{\frac{1}{2(1-\beta)}}\d r\Big)^{1-\beta}\\
				&\quad+C\|B^{H,i}\|_{S,T,\beta}\Big[ 1+(t-s)^{\beta}\|X^{i,N}\|_{s,t,\beta}+(t-s)^{\beta}\bigg(\frac{1}{N}\sum_{j=1}^N\|X^{j,N}\|^2_{s,t,\beta}\bigg)^{\frac{1}{2}}\Big]\\
				&\leq K\big(1+|X^{i,N}_s|+(\frac{1}{N}\sum_{j=1}^N|X^{j,N}_s|^2)^{\frac{1}{2}}\big)(t-s)^{1-\beta}+K\bigg(\int_{s}^{t}(r-s)^{\frac{\beta}{1-\beta}}\d r\bigg)^{1-\beta}\|X^{i,N}\|_{s,t,\beta}\\
				&\quad+K\Big(\int_{s}^{t}(r-s)^{\frac{2\beta}{2(1-\beta)}}\d r\Big)^{1-\beta}\bigg(\frac{1}{N}\sum_{j=1}^N\|X^{j,N}\|^2_{s,t,\beta}\bigg)^{\frac{1}{2}}\\
				&\quad+C\|B^{H,i}\|_{S,T,\beta}\Big[ 1+(t-s)^{\beta}\|X^{i,N}\|_{s,t,\beta}+(t-s)^{\beta}\bigg(\frac{1}{N}\sum_{j=1}^N\|X^{j,N}\|^2_{s,t,\beta}\bigg)^{\frac{1}{2}}\Big]\\
				&\leq C\Big(\|B^{H,i}\|_{S,T,\beta}+\bigg(1+|X^{i,N}_s|+\bigg(\frac{1}{N}\sum_{j=1}^N|X^{j,N}_s|^2\bigg)^{\frac{1}{2}}\bigg)(t-s)^{1-\beta}\Big)\\
				&\quad+C(t-s)^{\beta}\Big(\|B^{H,i}\|_{S,T,\beta}+(t-s)^{1-\beta}\Big)\|X^{i,N}\|_{s,t,\beta}\\
				&\quad+C(t-s)^{\beta}\Big(\|B^{H,i}\|_{S,T,\beta}+(t-s)^{1-\beta}\Big)\bigg(\frac{1}{N}\sum_{j=1}^N\|X^{j,N}\|^2_{s,t,\beta}\bigg)^{\frac{1}{2}}.
			\end{aligned}
		\end{equation}
		Note that since all $\{X^{j,N}_s\},$ $j=1,\cdot\cdot\cdot,N$ are independent and identically distributed, we obtain
		\begin{align*}
			\E\bigg(\bigg(\frac{1}{N}\sum_{j=1}^N|X^{j,N}_s|^2\bigg)^{\frac{1}{2}}\bigg)\leq \bigg(\frac{1}{N}\sum_{j=1}^N\E|X^{j,N}_s|^2\bigg)^{\frac{1}{2}}=(\E|X^{i,N}_s|^2)^{\frac{1}{2}}
		\end{align*}
		and
		\begin{equation}
			\begin{aligned}
				\mathbb{E}\bigg(\bigg(\frac{1}{N}\sum_{j=1}^{N}\|X^{j,N}\|_{s,t,\beta}^{2}\bigg)^{\frac{1}{2}}\bigg)& \leq\left(\frac{1}{N}\sum_{j=1}^{N}\E\|X^{j,N}\|_{s,t,\beta}^{2}\right)^{\frac{1}{2}} =\big(\mathbb{E}\|X^{i,N}\|_{s,t,\beta}^{2}\big)^{\frac{1}{2}}.
			\end{aligned}
		\end{equation}
		Taking the expectation of both sides of \eqref{eq-4.13},  we arrive at
		\begin{equation}\label{eq-3.19}
			\begin{aligned}
				\E\|X^{i,N}\|^p_{s,t,\beta}&\leq C\Big(\|B^{H,i}\|^p_{S,T,\beta}+\big(1+\E|X^{i,N}_s|^p+\E|X^{i,N}_s|^p\big)(t-s)^{p(1-\beta)}\Big)\\
				&\quad+C(t-s)^{p\beta}\Big(\|B^{H,i}\|^p_{S,T,\beta}+(t-s)^{p(1-\beta)}\Big)\E\|X^{i,N}\|^p_{s,t,\beta}\\
				&\quad+C(t-s)^{p\beta}\Big(\|B^{H,i}\|^p_{S,T,\beta}+(t-s)^{p(1-\beta)}\Big)\E\|X^{i,N}\|^p_{s,t,\beta}\\
				&\leq C\Big[\Big(\|B^{H,i}\|^p_{S,T,\beta}+\big(1+\E|X^{i,N}_s|^p\big)(t-s)^{p(1-\beta)}\Big)\\
				&\quad+(t-s)^{p\beta}\Big(\|B^{H,i}\|^p_{S,T,\beta}+(t-s)^{p(1-\beta)}\Big)\E\|X^{i,N}\|^p_{s,t,\beta}\Big].\\
			\end{aligned}
		\end{equation}	
		Without of lost generality, we assume that $C \geq 1$ in \eqref{eq-3.19}, and let
		\begin{equation}\label{eq-4.15}
			\Delta_1 = \left( \frac{1 \wedge (T-S)^{p\beta}}{6 C \| B^{H,i} \|^p _{S, T, \beta}} \right)^{\frac{1}{p\beta}} \wedge \left(\frac{1\wedge(T-S)^{p\beta}}{6C}\right)^\frac{1}{p}{},
		\end{equation}
		Taking $t = s + \Delta_1$, we get
		\begin{equation}
			\begin{aligned}
				\E\|X^{i,N}\|^p_{s,s + \Delta_1,\beta}
				&\leq\frac{C\Big(\|B^{H,i}\|^p_{S,T,\beta}+\big(1+\E|X^{i,N}_s|^p\big)\Delta_1^{p(1-\beta)}\Big) }{1-C\Delta_1^{p\beta}\Big(\|B^{H,i}\|^p_{S,T,\beta}+\Delta_1^{p(1-\beta)}\Big)}\\
				&\leq\frac{C\Big(\|B^{H,i}\|^p_{S,T,\beta}+\big(1+\E|X^{i,N}_s|^p\big)\Delta_1^{p(1-\beta)}\Big) }{1-\frac{1}{6}-\frac{1}{6}}\\
				&\leq 2C\Big(\|B^{H,i}\|^p_{S,T,\beta}+\big(1+\E|X^{i,N}_s|^p\big)\Delta_1^{p(1-\beta)}\Big).
			\end{aligned}
		\end{equation}
		Set $\eta_1:=2C\Delta_1^p$, for every $k\in N$,
		\begin{equation}
			\Pi_k:=\E\|X^{i,N}\|^p_{S+k\Delta_1,S+(k+1)\Delta_1\wedge T,\beta}\Delta_1^{p\beta}.
		\end{equation}
		So,we have
		\begin{equation}\label{eq-4.18}
			\begin{aligned}
				\Pi_k&\leq 2C\Big(\|B^{H,i}\|^p_{S,T,\beta}\Delta_1^{p\beta}+\Delta_1^p+\Delta_1^{p}\E|X^{i,N}_{S+k\Delta_1}|^p\Big)\\
				&\leq 2C\Big(\frac{1\wedge (T-S)^{p\beta}}{6C}+\frac{1\wedge (T-S)^{p\beta}}{6C}\Big)+\eta_1\E|X^{i,N}_{S+k\Delta_1}|^p\\
				&\leq 1\wedge (T-S)^{p\beta}+\eta_1\E|X^{i,N}_{S+k\Delta_1}|^p,
			\end{aligned}
		\end{equation}
		which leads to
		\begin{equation}\label{eq-4.19}
			\begin{aligned}
				\E|X^{i,N}_{S+k\Delta_1}|^p&\leq 2^{p-1}\E|X^{i,N}_{S+k\Delta_1}-X^{i,N}_{S+(k-1)\Delta_1}|^p+2^{p-1}\E|X^{i,N}_{S+(k-1)\Delta_1}|^p\\
				&\leq 2^{p-1}\Pi_{k-1}+2^{p-1}\E|X^{i,N}_{S+(k-1)\Delta_1}|^p\\
				&\leq 2^{p-1}\big(1\wedge (T-S)^{p\beta} \big)+2^{p-1}(1+\eta_1)\E|X^{i,N}_{S+(k-1)\Delta_1}|^p.
			\end{aligned}
		\end{equation}
		We can see that
		\begin{equation}
			\begin{aligned}
				\E|X^{i,N}_{S+k\Delta_1}|^p&\leq \big(2^{p-1}(1+\eta_1)\big)^k\E|X^{i,N}_S|^p+2^{p-1}\Big(\sum_{i=0}^{k-1}[2^{p-1}(1+\eta_1)]^i\Big)\big(1\wedge (T-S)^{p\beta} \big)\\
				&\leq 2^{k(p-1)}(1+\eta_1)^k\E|X^{i,N}_S|^p+2^{p-1}\frac{[2^{p-1}(1+\eta_1)]^k-1}{2^{p-1}(1+\eta_1)-1}\big(1\wedge (T-S)^{p\beta} \big).
			\end{aligned}
		\end{equation}
		Combining \eqref{eq-4.18}, we have
		\begin{align*}
			\Pi_k
			&\leq 2^{k(p-1)}\eta_1(1+\eta_1)^k\E|X^{i,N}_S|^p+\Big[\frac{2^{p-1}\eta_1(2^{k(p-1)}(1+\eta_1)^k-1)}{2^{p-1}(1+\eta_1)-1}+1\Big]\big(1\wedge (T-S)^{p\beta} \big)\\
			&\leq 2^{k(p-1)}(1+\eta_1)^k\Big(1\wedge (T-S)^{p\beta}+\eta_1\E|X^{i,N}_S|^p\Big).
		\end{align*}
		Notice that
			\begin{align*}
				&\E\|X^{i,N}\|^p_{S,T,\infty}\\
				&\leq
				\big([\frac{T-S}{\Delta_1}]+2 \big)^{p-1} \Big[\E|X^{i,N}_S|^p+\sum_{i=0}^{[\frac{T-S}{\Delta_1}]}\E\|X^{i,N}\|^p_{S+i\Delta_1,(S+(i+1)\Delta_1)\wedge T,\beta}\Delta_1^{p\beta}\Big]\\
				&\leq \big([\frac{T-S}{\Delta_1}]+2 \big)^{p-1}
				\Big[\E|X^{i,N}_S|^p+\sum_{i=0}^{[\frac{T-S}{\Delta_1}]}\Pi_i\Big]\\
				&\leq \big([\frac{T-S}{\Delta_1}]+2 \big)^{p-1}
				\Big[\E|X^{i,N}_S|^p+\sum_{i=0}^{[\frac{T-S}{\Delta_1}]}[2^{i(p-1)}(1+\eta_1)^i]\Big(1\wedge (T-S)^{p\beta}+\eta_1\E|X^{i,N}_S|^p\Big)\Big]\\
				&\leq \big([\frac{T-S}{\Delta_1}]+2 \big)^{p-1}
				\Big[\E|X^{i,N}_S|^p+\sum_{i=0}^{[\frac{T-S}{\Delta_1}]}2^{i(p-1)}\sum_{i=0}^{[\frac{T-S}{\Delta_1}]}(1+\eta_1)^i\Big(1\wedge (T-S)^{p\beta}+\eta_1\E|X^{i,N}_S|^p\Big)\Big]\\
				&\leq \big([\frac{T-S}{\Delta_1}]+2 \big)^{p-1}
				\Big[\E|X^{i,N}_S|^p+\Big[\frac{2^{[\frac{T-S}{\Delta_1}](p-1)}-1}{2^{p-1}-1}\Big]\\
				&\quad\quad\quad\quad\times\frac{(1+\eta_1)^{[\frac{T-S}{\Delta_1}]+1}-1}{\eta_1}\Big(1\wedge (T-S)^{p\beta}+\eta_1\E|X^{i,N}_S|^p\Big)\Big]\\
				&\leq \big([\frac{T-S}{\Delta_1}]+2 \big)^{p-1}
				\Big[\Big(\Big[\frac{2^{[\frac{T-S}{\Delta_1}](p-1)}-1}{2^{p-1}-1}\Big][(1+\eta_1)^{[\frac{T-S}{\Delta_1}]+1}-1]+1\Big)\E|X^{i,N}_S|^p\\
				&\quad+\Big[\frac{2^{[\frac{T-S}{\Delta_1}](p-1)}-1}{2^{p-1}-1}\Big]\frac{(1+\eta_1)^{[\frac{T-S}{\Delta_1}]+1}-1}{\eta_1}\Big(1\wedge (T-S)^{p\beta}\Big)\Big]\\
				&\leq \big([\frac{T-S}{\Delta_1}]+2 \big)^{p-1}
				\Big[\Big(\Big[\frac{2^{[\frac{T-S}{\Delta_1}](p-1)}-1}{2^{p-1}-1}\Big][(1+\eta_1)^{[\frac{T-S}{\Delta_1}]+1}-1]+1\big)\E|X^{i,N}_S|^p\\
				&\quad+\Big[\frac{2^{[\frac{T-S}{\Delta_1}](p-1)}-1}{2^{p-1}-1}\Big](1+\eta_1)^{[\frac{T-S}{\Delta_1}]}\Big([\frac{T-S}{\Delta_1}]+1\Big)\Big(1\wedge (T-S)^{p\beta}\Big)\Big].\\
			\end{align*}
		Noting that $\eta=2C\Delta_1^p\leq 1\wedge(T-S)^{p\beta}$ and $\eta_1^{-1}\log(1+\eta_1)\leq 1$, we find that
		\begin{align*}
			(1+\eta_1)^{[ \frac{T-S}{\Delta_1} ] + 1} & \leq \exp\left\{ \frac{T-S+\Delta_1}{\Delta_1} \log(1+\eta_1) \right\} \\
			& = \exp\left\{ (2C(T-S)\Delta_1^{p-1} + 2C\Delta_1^p) \frac{\log(1+\eta_1)}{\eta_1} \right\} \\
			& \leq \exp\left\{ C((T-S)\wedge (T-S)^{(p-1)\beta+1}) + (1\wedge(T-S)^{p\beta}) \right\} \\
			&:=G_1(T-S)
		\end{align*}
		and
		\begin{align*}
			(1+\eta_1)^{[ \frac{T-S}{\Delta_1} ]}\leq \exp\left\{ C((T-S)\wedge (T-S)^{(p-1)\beta+1})  \right\} :=G_2(T-S).
		\end{align*}
		By \eqref{eq-4.15}, we can see
		\begin{equation}
			\begin{aligned}
				\big[\frac{T-S}{\Delta_1}\big]\leq \frac{T-S}{\Delta_1}&\leq C\left\{(1\vee (T-S))\|B^{H,i}\|_{S,T,\beta}^{\frac{1}{\beta}}+((T-S)\vee (T-S)^{1-\beta})\right\}\\
				&:=G(T-S,B^{H,i}).
			\end{aligned}
		\end{equation}
		Since $\|B^{H}\|^q_{S,T,\beta}\leq C_{q,H,\beta}(T-S)^{q(H-\beta)}$for any $q \geq 1$ (see, e.g., \cite[Lemma 8]{Saussereau12}), so we get
		\begin{align*}
			G(T-S,B^{H,i})&\leq C\left\{(1\vee (T-S))C_{q,H,\beta}T^{\frac{H-\beta}{\beta}}+((T-S)\vee (T-S)^{1-\beta})\right\}\\
			&\leq C\left\{((T-S)^\frac{H-\beta}{\beta}\vee (T-S)^{\frac{H}{\beta}})+((T-S)\vee (T-S)^{1-\beta})\right\}\\
			&:=G_3(T-S).
		\end{align*}
		Thus, we obtain
		\begin{equation}
			\begin{aligned}
				&\E\|X^{i,N}\|^p_{S,T,\infty}\\
				&\leq \big(G_3(T-S)+2 \big)^{p-1}
				\Big[\Big(\Big[\frac{2^{G_3(T-S)(p-1)}-1}{2^{p-1}-1}\Big][G_1(T-S)-1]+1\big)\E|X^{i,N}_S|^p\\
				&\quad+\Big[\frac{2^{G_3(T-S)(p-1)}-1}{2^{p-1}-1}\Big]G_2(T-S)\Big(G_3(T-S)+1\Big)\Big(1\wedge (T-S)^{p\beta}\Big)\Big]\\
				&:=C_1.
			\end{aligned}
		\end{equation}
		Applying Lemma 5.2 in \cite{fan} with $n =1+[\frac{T-S}{\Delta}]$, we can see
		\begin{equation}\label{eq-3.29}
			\begin{aligned}
				&\E\|X^{i,N}\|^p_{S,T,\beta}\\
				&\leq \Big(1+\big[\frac{T-S}{\Delta_1}\big]\Big)^{p(1-\beta)}\max_{0\leq k \leq [\frac{T-S}{\Delta_1}]}\E\|X^{i,N}\|^p_{S+k\Delta_1,(S+(k+1)\Delta_1)\wedge T,\beta}\\
				&\leq \Big(1+\big[\frac{T-S}{\Delta_1}\big]\Big)^{p(1-\beta)}\Big(\Delta_1^{-p\beta}(1\wedge (T-S)^{p\beta})+2C\Delta_1^{p(1-\beta)}\max_{0\leq k \leq [\frac{T-S}{\Delta_1}]}\E|X^{i,N}_{S+k\Delta_1}|^p\Big)\\
				&\leq \Big(1+\frac{T-S}{\Delta_1}\Big)^{p(1-\beta)}\Big(\frac{T-S}{\Delta_1}\Big)^{p\beta}+C(\Delta_1+T-S)^{p(1-\beta)}\E\|X^{i,N}\|^p_{S,T,\infty}\\
				&\leq (1+G_3(T-S))^{p(1-\beta)}G_3(T-S)^{p\beta}\\
				&\quad+C\Big(\big(1\wedge (T-S)^{\beta}\big)+(T-S)\Big)^{p(1-\beta)}\E\|X^{i,N}\|^p_{S,T,\infty}\\
				&:=C_2.
			\end{aligned}
		\end{equation}
		Similarly, we can show
		\begin{align*}
			\E\|X^i\|^p_{S,T,\infty}+\E\|X^i\|^p_{S,T,\beta}< \infty.
		\end{align*}
	\end{proof}
	
		\begin{lemma}\label{lem3.3}
		Suppose that $\mu \in \mathscr P_{2,\beta}(W^d_{S,T})$, then it holds that
		\begin{equation}
			\E\W^2_{2,S,T,\beta}(\nu^N,\L(X^i))\leq C\epsilon_N,
		\end{equation}
			where
		\begin{align*}
			\epsilon_N=\left\{\begin{array}{ll}N^{-1/2}+N^{-(q-p)/q},&\text{if }p>d/2\text{ and }q\neq2p,\\N^{-1/2}\log(1+N)+N^{-(q-p)/q},&\text{if }p=d/2\text{ and }q\neq2p,\\N^{-p/d}+N^{-(q-p)/q},&\text{if }p\in[2,d/2)\text{ and }q\neq d/(d-p).\end{array}\right.
		\end{align*}
		\end{lemma}
		\begin{proof}
			Let p=2, for any $t\in [0,T]$ and $q>p$, according to lemma \ref{lem-3.2} and the Wasserstein distance estimate in Fournier and Guillin \cite[Theorem 1]{Fou},  there exists a constant $C$ such that
			\begin{equation}
				\E W^2_2(\L_{X^i_t},\nu_t^N)\leq C(\E|X^i_t|^q)^{\frac{2}{q}}\epsilon_N\leq C(\E\|X^i\|^q_{T,\infty})^{\frac{2}{q}}\epsilon_N\leq C\epsilon_N.
			\end{equation}
			By the definition of Wasserstein distance, there exists $\pi_t^*\in \mathscr C(\L_{X^i_t},\nu_t^N)$ such that
			\begin{equation}
				\int_{\R^d \times \R^d}|x-y|^2\pi_t^*(\d x,\d y)=W^2_2(\L_{X^i_t},\nu_t^N)
			\end{equation}
			for any $t\in [0,T]$. According to \eqref{AddHO1},
			we can see that for a certain point $t'$ in $[S,T]$
			\begin{align*}
				\E\W^2_{2,S,T}(\L_{X^i},\nu^N)
				&\leq \E\int_{W^d_{S,T}\times W^d_{S,T}}\|\gamma_1-\gamma_2\|^2_{S,T,\infty}\pi^*(\d \gamma_1,\d \gamma_2)\\
				&\leq \E\int_{W^d_{S,T}\times W^d_{S,T}}\big(\sup_{S\leq t \leq T}|\gamma_1(t)-\gamma_2(t)|\big)^2\pi^*(\d \gamma_1,\d \gamma_2)\\
				&\leq \E\int_{W^d_{S,T}\times W^d_{S,T}}\sup_{S\leq t \leq T}|\gamma_1(t)-\gamma_2(t)|^2\pi^*(\d \gamma_1,\d \gamma_2)\\
				&\leq \E\int_{W^d_{S,T}\times W^d_{S,T}}|\gamma_1(t')-\gamma_2(t')|^2\pi^*(\d \gamma_1,\d \gamma_2)\\
				&\leq\E W^2_2(\L_{X^i_{t'}},\nu_{t'}^N)\\
				&\leq \E\sup_{0\leq t\leq T}W^2_2(\L_{X^i_{t}},\nu_{t}^N).
			\end{align*}
			Therefore,
			\begin{equation}
				\E\lim_{N\to \infty}\W_{2,S,T}(\L_{X^i},\nu^N)=0.
			\end{equation}
			According to Lemma 3.2(2) in \cite{fan}, we obtain
			\begin{align*}
				\lim_{N\to \infty}\sup_{S\leq s_1<s_2\leq T}W^c_{2,s_1,s_2}(\nu^N,\L_{X^i})\leq \lim_{N\to \infty}2\W_{2,S,T}(\nu^N,\L_{X^i})=0.
			\end{align*}
			Let $\nu^M_t:=\frac{1}{M}\sum\limits_{j=1}^M\delta_{{X}^j_{t}}$ also be an empirical measure regarding $X^j_t$, we deduce that for every $\epsilon > 0$, there
			exists $N_0 \in \N$ such that
			\begin{align*}
				\E\W_{2,S,T,\beta}(\nu^M,\nu^N)\leq \epsilon, \quad N,M>N_0.
			\end{align*}
			By the triangle inequality,
			\begin{align*}
				\E W_{2,s_1,s_2}^c(\nu^N,\L_{X^i})&=\E\big( \lim_{M\to \infty}W_{2,s_1,s_2}^c(\nu^N,\L_{X^i})\big)\\
				&\leq \E\big( \lim_{M\to \infty}W_{2,s_1,s_2}^c(\nu^N,\nu^M)\big)+\E \big( \lim_{M\to \infty}W_{2,s_1,s_2}^c(\nu^M,\L_{X^i})\big)\\
				&\leq \E \big( \lim_{M\to \infty}W_{2,s_1,s_2}^c(\nu^N,\nu^M)\big).
			\end{align*}
			Hence,
			\begin{align*}
				\E\sup_{S\leq s_1<s_2\leq T}\frac{W^c_{2,s_1,s_2}(\nu^N,\L_{X^i})}{(s_2-s_1)^{\beta}}&\leq \E\sup_{S\leq s_1<s_2\leq T}\lim_{M\to \infty}\frac{W_{2,s_1,s_2}^c(\nu^N,\nu^M)}{(s_2-s_1)^{\beta}}\\
				&\leq \E\limsup_{M \to \infty}\sup_{S\leq s_1<s_2\leq T}\frac{W_{2,s_1,s_2}^c(\nu^N,\nu^M)}{(s_2-s_1)^{\beta}}\\
				&\leq \E\limsup_{M \to \infty}\W_{2,S,T,\beta}(\nu^M,\nu^N)\leq \epsilon,\quad N\geq N_0.
			\end{align*}
			Based on the arbitrariness of $\epsilon$, we can set $\epsilon=\sqrt{\epsilon_N}$.
			Combining the above arguments, we have
			\begin{equation}
				\E\W^2_{2,S,T,\beta}(\nu^N,\L(X^i))\leq C\epsilon_N.
			\end{equation}
		\end{proof}

	\begin{theorem}\label{thm-3.4}
		Suppose that $\frac{1}{2}<\beta<H$ and Assumption \ref{ass-2.3} hold, then when $T-S\leq R_0$, it holds that
		\begin{equation}
			\begin{aligned}
			\max_{1\leq k \leq N}{\E\|X^i-X^{i,N}\|^2_{S,T,\infty}}+\max_{1\leq k \leq N}{\E\|X^i-X^{i,N}\|^2_{S,T,\beta}}+\E\W^2_{2,S,T,\beta}(\mu^N,\L(X^i))\leq C\epsilon_N.
			\end{aligned}
		\end{equation}
	   $R_0$ will be given in the proof.
	\end{theorem}
	\begin{proof}
		It follows from \eqref{eq-4.1} and \eqref{eq-4.2} that
		\begin{equation}
			\begin{aligned}
				X^{i}_t-X^{i,N}_t=\int_{S}^{t}b(X^{i}_s,\L_{X^i_s})-b(X^{i,N}_s,{\mu}^{N}_s)\d s+\int_{S}^{t}\sigma(X^{i}_s,\L_{X^i_s})-\sigma(X^{i,N}_s,{\mu}^{N}_s)\d B^{H,i}_s.
			\end{aligned}
		\end{equation}
		For any $ S\leq s<t \leq T$ and $t-s\leq 1$, we have
		\begin{equation}
			\begin{aligned}
				&X^{i}_t-X^{i,N}_t-(X^{i}_s-X^{i,N}_s)\\
				&=\int_{s}^{t}b(X^{i}_r,\L_{X^i_r})-b(X^{i,N}_r,{\mu}^{N}_r)\d r+\int_{s}^{t}\sigma(X^{i}_r,\L_{X^i_r})-\sigma(X^{i,N}_r,{\mu}^{N}_r)\d B^{H,i}_r.
			\end{aligned}
		\end{equation}
		By Assumption \ref{ass-2.3},
		\begin{align*}
			\int_{s}^{t}b(X^{i}_r,\L_{X^i_r})-b(X^{i,N}_r,{\mu}^{N}_r)\d B^{H,i}_r&\leq K\int_{s}^{t}(|X^{i}_r-X^{i,N}_r|+W_2(\L_{X^i_r}, \mu_r^{N}))\d r\\
			&\leq K(\|X^{i}-X^{i,N}\|_{s,t,\infty}+\W_{2,S,T,\beta}(\L_{X^i_r}, \mu^{N}))(t-s),
		\end{align*}
		combining this with Lemma \ref{lem-2.5}, we get
		\begin{align*}
			&\int_{s}^{t}\sigma(X^{i}_r,\L_{X^i_r})-\sigma(X^{i,N}_r,{\mu}^{N}_r)\d B^{H,i}_r\\
			\le&\|B^{H,i}\|_{s,t,\beta_1} \left(\Lambda_1 (t-s)^{\beta_1}+\Lambda_2  (t-s)^{\alpha+\beta_1}\right)\|X^i-X^{i,N}\|_{s,t,\infty}\cr
			&+\Lambda_3 \|B^{H,i}\|_{s,t,\beta_1} (t-s)^{\beta+\beta_1}\|X^i-X^{i,N}\|_{s,t,\beta}\cr
			&+\|B^{H,i}\|_{s,t,\beta_1} \left(\Lambda_4 (t-s)^{\beta_1}+\Lambda_5 (t-s)^{\alpha+\beta_1}\right)\W_{2,S,T,\beta}(\L_{X^i}, \mu^{N})).
		\end{align*}
		Hence,
		\begin{align*}
			&\|X^{i}-X^{i,N}\|_{s,t,\beta}\\
			\le& \,\Lambda_3\|B^{H,i}\|_{s,t,\beta_1} (t-s)^{ \beta_1}\|X^i-X^{i,N}\|_{s,t,\beta}\cr
			&+\left[\|B^{H,i}\|_{s,t,\beta_1} (\Lambda_1(t-s)^{\beta_1-\beta}+\Lambda_2(t-s)^{\alpha+\beta_1-\beta})+K(t-s)^{1-\beta}\right]\|X^i-X^{i,N}\|_{s,t,\infty}\cr
			&+\left[\|B^{H,i}\|_{s,t,\beta_1} (\Lambda_4(t-s)^{\beta_1-\beta}+\Lambda_5(t-s)^{\alpha+\beta_1-\beta})+K(t-s)^{1-\beta}\right]\W_{2,S,T,\beta}(\L_{X^i}, \mu^{N})).
		\end{align*}
By   \cite[Proposition 3.8]{fan} and $x\leq \e^x$ for $x>0$, we have
\begin{align}\label{2PfLeContr}
	\|X^i-X^{i,N}\|_{S,T,\infty}
	& \leq 2 (1\wedge (T-S)^{\beta_1})e^{(T-S)\Delta_2^{-1}\log 2}\W_{2,S,T,\beta}(\L(X^i),\mu^N)
\end{align}
and
\begin{align*}
\|X^i-X^{i,N}\|_{S,T,\beta}&\leq  2(1\wedge (T-S)^{\beta_1-\beta}) \left((T-S)\Delta_2^{-1}\right)\left(  e^{(T-S)\Delta_2^{-1}\log 2}+\frac { 1} {2}\right)\W_{2,S,T,\beta}(\L(X^i),\mu^N)\\
&\leq 4 (1\wedge (T-S)^{\beta_1})e^{(T-S)\Delta_2^{-1}\log 2}\W_{2,S,T,\beta}(\L(X^i),\mu^N),
\end{align*}
where
\begin{align}\label{1LeContr}
\Delta_2:=&\left(\frac { 1  \wedge (T-S)^{\beta_1}} {3(3\Lambda_1\vee \Lambda_3\vee3\Lambda_4 )\|B^{H,i}\|_{S,T,\beta_1}}\right)^{\frac 1 {\beta_1}}\wedge\left(\frac {1\wedge (T-S)^{\beta_1}} {9 (\Lambda_2\vee \Lambda_5)\|B^{H,i}\|_{S,T,\beta_1}}\right)^{\frac 1 {\alpha+\beta_1}}\wedge\left(\frac {1\wedge (T-S)^{\beta_1}} {9 (1\vee K)}\right)^{\frac 1 {\beta_1}}
\end{align}
with $\Lambda_1,\cdots,\Lambda_5$ giving by \eqref{1-EsNoi}.
	Notice that $1\wedge (T-S)^{\beta_1}\leq 1\wedge (T-S)^{\beta_1-\beta}$ and $x\leq \e^x$ for $x>0$. Let $C(T-S)=(1\wedge (T-S)^{\beta_1-\beta})\left\|e^{(T-S)\Delta_2^{-1}\log 2}\right\|_{L^2(\Omega)}$, \cite{fan} tell us that $\lim\limits_{T-S \to 0^+}C(T-S)=0$ and using Lemma \ref{lem3.1}, we get
	\begin{align*}
		&\sqrt{\E\|X^i-X^{i,N}\|^2_{S,T,\infty}}\\
		&\leq 2(1\wedge (T-S)^{\beta_1-\beta})  \left\|e^{(T-S)\Delta_2^{-1}\log 2}\right\|_{L^2(\Omega)}\E\W_{2,S,T,\beta}(\L(X^i),\mu^N)\\
		&\leq 2C(T-S)  \Big(\E\W_{2,S,T,\beta}(\L(X^i),\nu^N)+\E\W_{2,S,T,\beta}(\nu^N,\mu^N)\Big)\\
		&\leq 2C(T-S)  \Big(\E\W_{2,S,T,\beta}(\L(X^i),\nu^N)+\E\sqrt{\frac{1}{N}\sum_{j=1}^N\|X^i-X^{i,N}\|^2_{S,T,\infty}}+\E\sqrt{\frac{1}{N}\sum_{j=1}^N\|X^i-X^{i,N}\|^2_{S,T,\beta}}\Big)\\
		&\leq 2C(T-S)\Big(\E\W_{2,S,T,\beta}(\L(X^i),\nu^N)+\sqrt{\E\|X^i-X^{i,N}\|^2_{S,T,\infty}}+\sqrt{\E\|X^i-X^{i,N}\|^2_{S,T,\beta}}\Big),
	\end{align*}
	and
	\begin{align*}
		&\sqrt{\E\|X^i-X^{i,N}\|^2_{S,T,\beta}}\\
		&\leq 2(1\wedge (T-S)^{\beta_1-\beta}) \left\|(T-S)\Delta_2^{-1}\left(e^{(T-S)\Delta_2^{-1}\log 2}+1\right)\right\|_{L^2(\Omega)}\W_{2,S,T,\beta}(\L(X^i),\mu^N)\\
		&\leq 4(1\wedge (T-S)^{\beta_1-\beta})  \left\|e^{(T-S)\Delta_2^{-1}\log 2}\right\|_{L^2(\Omega)}\W_{2,S,T,\beta}(\L(X^i),\mu^N)\\
		&\leq 4C(T-S)\Big(\E\W_{2,S,T,\beta}(\L(X^i),\nu^N)+\sqrt{\E\|X^i-X^{i,N}\|^2_{S,T,\infty}}+\sqrt{\E\|X^i-X^{i,N}\|^2_{S,T,\beta}}\Big).
	\end{align*}
	Let\begin{align*}
		R_0=\inf\Big\{ R>0|C(R)\geq \frac{1}{8}\Big\}.
	\end{align*}
	we can see that for any $T-S\leq R_0$, we have
	\begin{equation}
		\sqrt{\E\|X^i-X^{i,N}\|^2_{S,T,\infty}}\leq \E\W_{2,S,T,\beta}(\L(X^i),\nu^N),
	\end{equation}
	and \begin{equation}
		\sqrt{\E\|X^i-X^{i,N}\|^2_{S,T,\beta}}\leq 2\E\W_{2,S,T,\beta}(\L(X^i),\nu^N).
	\end{equation}
	Through trigonometric inequality, we have
	\begin{equation}
		\begin{aligned}
			&\E\W_{2,S,T,\beta}(\L(X^i),\mu^N)\\
			&\leq \E\W_{2,S,T,\beta}(\L(X^i),\nu^N)+\sqrt{\E\|X^i-X^{i,N}\|^2_{S,T,\infty}}+\sqrt{\E\|X^i-X^{i,N}\|^2_{S,T,\beta}}\\
			&\leq 4\E\W_{2,S,T,\beta}(\L(X^i),\nu^N).
		\end{aligned}
	\end{equation}
	By Lemma \ref{lem3.3}, we can conclude that:
		\begin{equation}
		\begin{aligned}
			{\E\|X^i-X^{i,N}\|^2_{S,T,\infty}}+{\E\|X^i-X^{i,N}\|^2_{S,T,\beta}}+\E\W^2_{2,S,T,\beta}(\mu^N,\L(X^i))\leq C\epsilon_N.
		\end{aligned}
	\end{equation}
\end{proof}

\section{EM method}
In this section, we establish the classical EM method for the interacting particle system \eqref{eq-4.1}. For $n\geq 1$ large enough such that $\Delta=T/n\in (0,1]$ and $k=0,1, ..., n-1$, define $t_k=k\Delta$. The numerical solutions are then generated by the EM method
\begin{equation}
	Z_{t_{k+1}}^{i,N}=Z_{t_k}^{i,N}+b(Z_{t_k}^{i,N},\mu_{t_k}^{Z,N})\Delta+\sigma(Z_{t_k}^{i,N},\mu_{t_k}^{Z,N})\Delta B^{H,i}_{t_k},
\end{equation}
where $Z^{i,N}_0=X^{i}_0$, $\Delta B^{H,i}_{t_k}=B^{H,i}_{t_{k+1}}-B^{H,i}_{t_k}$ and the empirical measures $\mu_{t_k}^{Z,N}(\cdot)=\frac{1}{N}\sum_{j=1}^{N}\delta_{Z_{t_k}^{j,N}}(\cdot)$. The first is the piecewise constant extension given by
\begin{equation}
	\bar{Z}_t^{i,N}=Z_{t_k}^{i,N}, t_k\leq t \leq t_{k+1}.
\end{equation}
and the second is the continuous extension of the EM method defined by
\begin{equation}\label{eq-5.3}
	Z_{t}^{i,N}=\bar Z_{t_k}^{i,N}+\int_{t_k}^{t}b(\bar Z_{s}^{i,N},\bar \mu_{s}^{Z,N})\d s+\int_{t_k}^{t}\sigma(\bar Z_{s}^{i,N},\bar \mu_{s}^{Z,N})\d B^{H,i}_{s}.
\end{equation}
Here, the initial value $\bar{Z}_0^{i,N}=X^{i}_0$ and $\bar \mu_{t}^N(\cdot)=\frac{1}{N}\sum_{j=1}^{N}\delta_{\bar Z_{t}^{j,N}}(\cdot)$. for all $t \in [0, T ]$, we have
\begin{equation}\label{eq-5.4}
	Z_{t}^{i,N}=\bar Z_{0}^{i,N}+\int_{0}^{t}b(\bar Z_{s}^{i,N},\bar \mu_{s}^{Z,N})\d s+\int_{0}^{t}\sigma(\bar Z_{s}^{i,N},\bar \mu_{s}^{Z,N})\d B^{H,i}_{s}.
\end{equation}

\begin{lemma}\label{lem-4.1}
	Let $\frac{1}{2}<\beta<H$, Assumption \ref{ass-2.3} hold. Similar to the proof of Lemma \ref{lem-3.2}, we can show
	\begin{equation}
		\E\|Z^{i,N}\|^2_{T,\infty}+\E\|Z^{i,N}\|^2_{T,\beta}< \infty.
	\end{equation}
\end{lemma}

\begin{lemma}\label{lem-4.2}
	Let $\frac{1}{2}<\beta<H$, Assumption \ref{ass-2.3} hold, then
	\begin{equation}
		\E\|Z^{i,N}-\bar{Z}^{i,N}\|^2_{t_k,t_{k+1},\infty}\leq G_1(\Delta),
	\end{equation}
	\begin{equation}
		\E\|Z^{i,N}-\bar{Z}^{i,N}\|^2_{t_k,t_{k+1},\beta}\leq G_2(\Delta).
	\end{equation}
\end{lemma}
\begin{proof}
From \eqref{eq-5.3}, for any $t_k\leq s < t \leq t_{k+1}$, we have for $t_{k+1}-t_k=\Delta$,
\begin{equation}
	\begin{aligned}
		&Z_{t}^{i,N}-\bar Z_{t_k}^{i,N}-(Z_{s}^{i,N}-\bar Z_{s_k}^{i,N})\\
		&\leq \int_{s}^{t}b(\bar Z_{r}^{i,N},\bar \mu_{r}^{Z,N})\d r+\int_{s}^{t}\sigma(\bar Z_{r}^{i,N},\bar \mu_{r}^{Z,N})\d B^{H,i}_{r}.
	\end{aligned}
\end{equation}
Similarly to \eqref{eq-3.19},  we can obtain
\begin{equation}\label{eq-4.9}
	\begin{aligned}
		&\E\|Z^{i,N}-\bar{Z}^{i,N}\|^2_{s,t,\beta}\\
		&\leq C\Big(\|B^{H,i}\|^2_{t_k,t_{k+1},\beta}+\big(1+\E|\bar Z^{i,N}_s|^2\big)(t-s)^{2-2\beta}\Big)\\
		&\quad+C(t-s)^{2\beta}\Big(\|B^{H,i}\|^2_{t_k,t_{k+1},\beta}+(t-s)^{2-2\beta}\Big)\E\|\bar Z^{i,N}\|^2_{s,t,\beta}.
	\end{aligned}
\end{equation}	
Without of lost generality, we assume that $C \geq 1$ in \eqref{eq-4.9}, and let
\begin{equation}
	\Delta_3=\Big(\frac{ \Delta^{2\beta}}{6C\|B^{H,i}\|^2_{t_k,t_{k+1},\beta}}\Big)^{\frac{1}{2\beta}}\wedge \Big(\frac{ \Delta^{2\beta}}{6C}\Big)^{\frac{1}{2}}\wedge \Big(\frac{ \Delta^{2\beta}}{6C}\Big)^{\frac{1}{2\beta+2}}\wedge \Big(\frac{ \Delta^{2\beta}}{6C\|B^{H,i}\|^2_{t_k,t_{k+1},\beta}}\Big)^{\frac{1}{4\beta}}.
\end{equation}
Taking $t=s+\Delta_3$ and by combining with Lemma \ref{lem-4.1}, we have
\begin{equation}
	\begin{aligned}
		&\E\|Z^{i,N}-\bar{Z}^{i,N}\|^2_{s,s+\Delta_3,\beta}\\
		&\leq C\Big(\|B^{H,i}\|^2_{t_k,t_{k+1},\beta}+\big(1+\E\|\bar Z^{i,N}\|^2_{T,\infty}\big)\Delta_3^{2-2\beta}\Big)+C\Delta_3^{2\beta}\Big(\|B^{H,i}\|^2_{t_k,t_{k+1},\beta}+\Delta_3^{2-2\beta}\Big)\E\|\bar Z^{i,N}\|^2_{T,\beta}\\
		&\leq C\Big(\|B^{H,i}\|^2_{t_k,t_{k+1},\beta}+\Delta_3^{2-2\beta}+\Delta_3^{2\beta}\|B^{H,i}\|^2_{t_k,t_{k+1},\beta}+\Delta_3^2\Big).
	\end{aligned}
\end{equation}	
For every $l\in N$,
\begin{align*}
	\Pi_{2,l}&:=\E\|Z^{i,N}-\bar{Z}^{i,N}\|^2_{t_k+l\Delta_3,(t_k+(l+1)\Delta_3)\wedge t_{k+1},\beta}\Delta_3^{2\beta}\\
	&\leq C\Big(\Delta_3^{2\beta}\|B^{H,i}\|^2_{t_k,t_{k+1},\beta}+\Delta_3^{2}+\Delta_3^{4\beta}\|B^{H,i}\|^2_{t_k,t_{k+1},\beta}+\Delta_3^{2\beta+2}\Big)\\
	&\leq \Delta^{2\beta}.
\end{align*}
Consequently, it follows that
	\begin{equation}
	\begin{aligned}
		&\E\|Z^{i,N}-\bar Z^{i,N}\|^2_{t_k,t_{k+1},\infty}\\
		&\leq
		\big([\frac{\Delta}{\Delta_3}]+2 \big) \Big[\E|Z^{i,N}_{t_k}-\bar{Z}^{i,N}_{t_k}|^2+\sum_{l=0}^{[\frac{\Delta}{\Delta_3}]}\E\|X^{i,N}\|^2_{t_k+l\Delta_3,(t_k+(l+1)\Delta_3)\wedge t_{k+1},\beta}\Delta_3^{2\beta}\Big]\\
		&\leq \big([\frac{\Delta}{\Delta_3}]+2 \big)
		\Big[\sum_{l=0}^{[\frac{\Delta}{\Delta_3}]}\Pi_{2,l}\Big]\\
		&\leq \big([\frac{\Delta}{\Delta_3}]+2 \big)\big([\frac{\Delta}{\Delta_3}]+1 \big)\Delta^{2\beta}.
	\end{aligned}
\end{equation}
Notice that
\begin{equation}
	\begin{aligned}
		&[\frac{\Delta}{\Delta_3}]\leq \frac{\Delta}{\Delta_3}\\
		&\leq C\Bigg\{\Big\|B^{H,i}\|^{\frac{1}{\beta}}_{t_k,t_{k+1},\beta}\vee \Delta^{1-\beta}
		\vee  \Delta^{\frac{1}{1+\beta}}\vee \Big( \Delta^{\frac{1}{2}}\|B^{H,i}\|^{\frac{1}{2\beta}}_{t_k,t_{k+1},\beta}\Big)\Bigg\}\\
		&\leq  C\Bigg(\Delta^{\frac{H-\beta}{\beta}}+  \Delta^{1-\beta}+  \Delta^{\frac{1}{1+\beta}}+  \Delta^{\frac{H}{2\beta}}\Bigg)\\
		&\leq C\Big(\Delta^{\frac{H-\beta}{\beta}}\vee \Delta^{1-\beta}\Big).
	\end{aligned}
\end{equation}
We can see that for $\Delta<T$,
\begin{equation}
	\begin{aligned}
			\E\|Z^{i,N}-\bar Z^{i,N}\|^2_{t_k,t_{k+1},\infty}&\leq \big([\frac{\Delta}{\Delta_3}]+2 \big)\big([\frac{\Delta}{\Delta_3}]+1 \big)\Delta^{2\beta}\\
			&\leq C\Delta^{2\beta}.
	\end{aligned}
\end{equation}
Consistent with \eqref{eq-3.29}, we  obtain
\begin{equation}
	\begin{aligned}
		&\E\|Z^{i,N}-\bar Z^{i,N}\|^2_{t_k,t_{k+1},\beta}\\
		&\leq \Big(1+\big[\frac{\Delta}{\Delta_3}\big]\Big)^{2-2\beta}\max_{0\leq l \leq [\frac{\Delta}{\Delta_3}]}\E\|Z^{i,N}-\bar Z^{i,N}\|^2_{t_k+l\Delta_3,(t_k+(l+1)\Delta_3)\wedge t_{k+1},\beta}\\
		&\leq \Big(1+\big[\frac{\Delta}{\Delta_3}\big]\Big)^{2-2\beta}\Delta^{2\beta}\Delta_3^{-2\beta}\\
		&\leq \Big(1+\frac{\Delta}{\Delta_3}\Big)^{2-2\beta}\big(\frac{\Delta}{\Delta_3}\big)^{2\beta}\\
		&\leq C\Big(\Delta^{2(H-\beta)}\vee \Delta^{2\beta(1-\beta)}\Big).
	\end{aligned}
\end{equation}		
\end{proof}

\begin{theorem}\label{thm-4.3}
	Let $\frac{1}{2}<\beta<H$ and Assumption \ref{ass-2.3} hold. Then, for any interval $[t_k,t_{k+1}]$,
	\begin{equation}
		\begin{aligned}
			\sqrt{\E\|X^{i,N}- Z^{i,N}\|^2_{t_k,t_{k+1},\infty}}
			&\leq C\Delta^{(H-\beta)\wedge (\beta-\beta^2)},
		\end{aligned}
	\end{equation}
	and \begin{equation}
		\begin{aligned}
			\sqrt{\E\|X^{i,N}- Z^{i,N}\|^2_{t_k,t_{k+1},\beta}}
			&\leq C\Delta^{(H-\beta)\wedge (\beta-\beta^2)}.
		\end{aligned}
	\end{equation}
	\begin{proof}
By the same method as Theorem \ref{thm-3.4}, we can obtain that:
		\begin{align*}
			\sqrt{\E\|X^{i,N}- Z^{i,N}\|^2_{S,T,\infty}}\leq 2(1\wedge (T-S)^{\beta_1-\beta})
			\|\e^{(1+\log 2)(T-S)\Delta_2^{-1}}\|_{L^2(\Omega)}\W_{2,S,T,\beta}(\mu^N,\bar \mu^{Z,N}),
		\end{align*}
		and
		\begin{align*}
			\sqrt{\E\|X^{i,N}- Z^{i,N}\|^2_{S,T,\beta}}\leq 4(1\wedge (T-S)^{\beta_1-\beta})
			\|\e^{(1+\log 2)(T-S)\Delta_2^{-1}}\|_{L^2(\Omega)}\W_{2,S,T,\beta}(\mu^N,\bar \mu^{Z,N}).
		\end{align*}
		According to \eqref{3.10}, we can deduce that
		\begin{align*}
			&\W_{2,S,T,\beta}(\mu^N,\bar \mu^{Z,N})\\
			&\leq \W_{2,S,T,\beta}(\mu^N, \mu^{Z,N})+\W_{2,S,T,\beta}(\mu^{Z,N},\bar \mu^{Z,N})\\
			&\leq \sqrt{\E\|X^{i,N}- Z^{i,N}\|^2_{S,T,\infty}}+\sqrt{\E\|Z^{i,N}-\bar Z^{i,N}\|^2_{S,T,\infty}}\\
			&\quad+\sqrt{\E\|X^{i,N}- Z^{i,N}\|^2_{S,T,\beta}}+\sqrt{\E\|Z^{i,N}-\bar Z^{i,N}\|^2_{S,T,\beta}}.
		\end{align*}
		Using the same $R_0$ as in Theorem \ref{thm-3.4},	 for any $T-S\leq R_0$, we have
		\begin{equation}
			\sqrt{\E\|X^{i,N}- Z^{i,N}\|^2_{S,T,\infty}}\leq \sqrt{\E\|Z^{i,N}-\bar Z^{i,N}\|^2_{S,T,\infty}}+\sqrt{\E\|Z^{i,N}-\bar Z^{i,N}\|^2_{S,T,\beta}},
		\end{equation}
		and \begin{equation}
			\sqrt{\E\|X^{i,N}- Z^{i,N}\|^2_{S,T,\beta}}\leq 2\sqrt{\E\|Z^{i,N}-\bar Z^{i,N}\|^2_{S,T,\infty}}+2\sqrt{\E\|Z^{i,N}-\bar Z^{i,N}\|^2_{S,T,\beta}}.
		\end{equation}
		We can let $\Delta<R_0$. Thus, on an arbitrary interval $[t_k,t_{k+1}]$,
		\begin{equation}
			\begin{aligned}
				\sqrt{\E\|X^{i,N}- Z^{i,N}\|^2_{t_k,t_{k+1},\infty}}&\leq \sqrt{\E\|Z^{i,N}-\bar Z^{i,N}\|^2_{t_k,t_{k+1},\infty}}+\sqrt{\E\|Z^{i,N}-\bar Z^{i,N}\|^2_{t_k,t_{k+1},\beta}}\\
				&\leq C\Big( \Delta^{\beta}+\Big(\Delta^{H-\beta}\vee \Delta^{\beta(1-\beta)}\Big)\Big)\\
				&\leq C\Delta^{(H-\beta)\wedge (\beta-\beta^2)},
			\end{aligned}
		\end{equation}
		and \begin{equation}
			\begin{aligned}
				\sqrt{\E\|X^{i,N}- Z^{i,N}\|^2_{t_k,t_{k+1},\beta}}&\leq 2\sqrt{\E\|Z^{i,N}-\bar Z^{i,N}\|^2_{t_k,t_{k+1},\infty}}+2\sqrt{\E\|Z^{i,N}-\bar Z^{i,N}\|^2_{t_k,t_{k+1},\beta}}\\
				&\leq C\Delta^{(H-\beta)\wedge (\beta-\beta^2)}.
			\end{aligned}
		\end{equation}
			
	\end{proof}
\end{theorem}

Through trigonometric inequality, Theorem \ref{thm-3.4} and Theorem \ref{thm-4.3}, one can obtain the following convergence result of the proposed numerical method for the \eqref{eq-1}.

\begin{theorem}
	It holds under $\frac{1}{2}<\beta<H$, Assumption \ref{ass-2.3} and for an arbitrary interval $[t_k,t_{k+1}]$ that
	\begin{equation}
		\sup_{i\in \{1,...,N\}}\E\|X^i-Z^{i,N}\|^2_{t_k,t_{k+1},\infty}\leq C\Big( \epsilon_N+\Delta^{2(H-\beta)\wedge 2(\beta-\beta^2)}\Big),
	\end{equation}
	and
	\begin{equation}
		\sup_{i\in \{1,...,N\}}\E\|X^i-Z^{i,N}\|^2_{t_k,t_{k+1},\beta}\leq C \big( \epsilon_N+\Delta^{2(H-\beta)\wedge 2(\beta-\beta^2)}\big).
	\end{equation}
\end{theorem}

\section{Numerical examples}
\begin{example}
	In this section, we perform numerical experiments to verify our theoretical results. Consider the following one-dimensional mean field model driven by fBm,
	\begin{align}\label{5.1}
		\begin{aligned}
			\d {X}_t&=\Big(X_t+\int_\R z\mu(\d z)\Big)\d t+\Big( \sin\big( X_t+\int_\R z\mu(\d z)\big)\Big)\d B^{H}_t
		\end{aligned}
	\end{align}
	for $H\in (\frac{\sqrt{5}-1}{2},1)$, where $\mu=\L(X_t)$. Define $b(x,\mu)=x+\int_{0}^{t}z\mu(\d z)$ and $\sigma(x,\mu)=\sin(x+\int_{0}^{t}z\mu(\d z))$.
	\begin{align*}
		|b(x,\mu)-b(y,\nu)|&\leq \big|x+\int_\R z\mu(\d z)-y-\int_\R z'\nu(\d z')\big|\\
		&\leq \big|(x-y)+\int_\R z\mu(\d z)-\int_\R z'\nu(\d z')\big|\\
		&\leq \big|(x-y)+\int_\R |z-z'|\pi(\d z,\d z')\big|\\
		&\leq |x-y|+W_2(\mu,\nu).
	\end{align*}
	 We compute the derivatives of $\sigma(x,\mu)$.
	\begin{align*}
		&\frac{\d}{\d x}\sigma(x,\mu)=\cos\big(  x+\int_\R z\mu(\d z)\big),\quad
        D^L\sigma(x,\mu)(y)=\cos\big(  x+\int_\R z\mu(\d z)\big)y,\\
		&\frac{\d^2}{\d x^2}\sigma(x,\mu)=-\sin\big(  x+\int_\R z\mu(\d z)\big), \quad D^L(\frac{\d}{\d x}\sigma(x,\mu))(y)=-\sin\big(  x+\int_\R z\mu(\d z)\big)y\\
		&\frac{\d}{\d x}D^L\sigma(x,\mu)(y)=-\sin\big(  x+\int_\R z\mu(\d z)\big)y, \quad \frac{\d}{\d y}D^L\sigma(x,\mu)(y)=\cos\big(  x+\int_\R z\mu(\d z)\big),\\
		& D^{L,2}\sigma(x,\mu)(y,z)=-\sin\big(  x+\int_\R z\mu(\d z)\big)yz.
	\end{align*}
	It can be verified that these derivatives are all continuous and bounded. Thus, Assumption \ref{ass-2.3} is satisfied.
	The corresponding interacting particle system is written as
	\begin{align}\label{5.2}
		\begin{aligned}
			\d X^{i,N}_t
			=\Big(X^{i,N}_t+\frac{1}{N}\sum_{j=1}^NX^{j,N}_t\Big)\d t+
			\Big( \sin\big(X^{i,N}_t+\frac{1}{N}\sum_{j=1}^NX^{j,N}_t\big)\Big)\d B^{H,i}_t,
		\end{aligned}
	\end{align}
	where $i\in 1,...,N$. The EM scheme for interacting particle system reads
	\begin{align}\label{5.3}
		\begin{aligned}
			\d Z^{i,N}_{t_{k+1}}
			=\Big(Z^{i,N}_{t_k}+\frac{1}{N}\sum_{j=1}^NZ^{j,N}_{t_k}\Big)\Delta+
			\Big( \sin\big(Z^{i,N}_{t_k}+\frac{1}{N}\sum_{j=1}^NZ^{j,N}_{t_k}\big)\Big)\Delta B^{H,i}_{t_k}.
		\end{aligned}
	\end{align}
	By Theorem \ref{thm-3.4},  for $H=0.7,\beta=0.6$; $H=0.8,\beta=0.7$; $H=0.9,\beta=0.8$ and $T-S\leq R_0$, we can obtain
		\begin{equation}
		\begin{aligned}
			\max_{1\leq k \leq N}{\E\|X^i-X^{i,N}\|^2_{S,T,\infty}}+\max_{1\leq k \leq N}{\E\|X^i-X^{i,N}\|^2_{S,T,\beta}}\leq C\epsilon_N.
		\end{aligned}
	\end{equation}
Similarly, by Theorem \ref{thm-4.3}, we define the error
\begin{align*}
	E_1 =\big(\E\|X^i-Z^{i,N}\|^2_{t_k,t_{k+1},\infty}\big)^{\frac{1}{2}}\leq C\sqrt{ \epsilon_N+\Delta^{2(H-\beta)\wedge 2(\beta-\beta^2)}},
\end{align*}
and
\begin{align*}
	E_2=\big(\E\|X^i-Z^{i,N}\|^2_{t_k,t_{k+1},\beta} \big)^\frac{1}{2}\leq C\sqrt{ \epsilon_N+\Delta^{2(H-\beta)\wedge 2(\beta-\beta^2)}}.
\end{align*}
\end{example}

\begin{figure}[htbp]
\centering
\begin{minipage}[t]{0.3\linewidth}
\centering
\includegraphics[height=4cm, width=5.2cm]{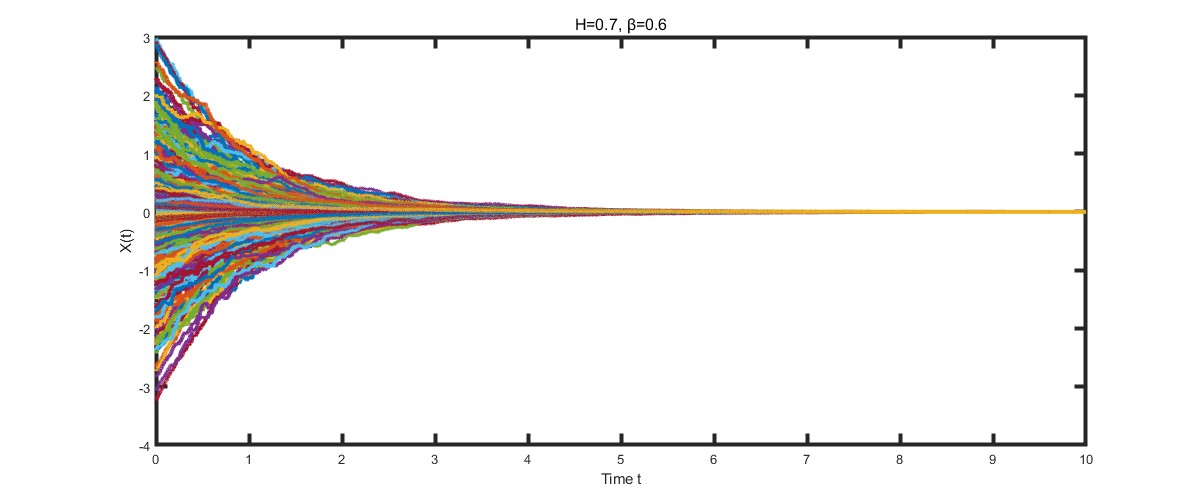}
\end{minipage}%
\begin{minipage}[t]{0.3\linewidth}
\centering
\includegraphics[height=4cm, width=5.2cm]{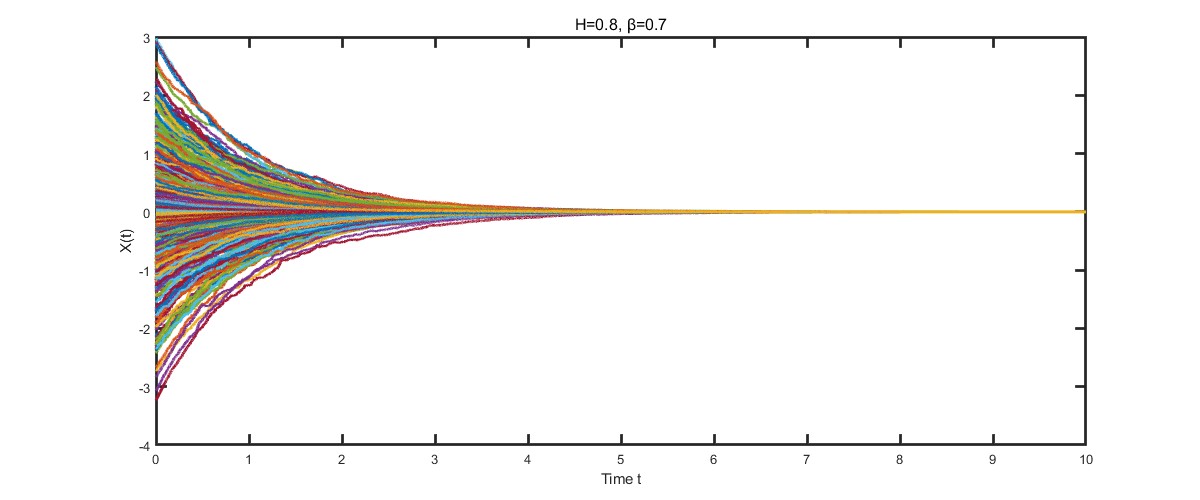}
\end{minipage}%
\begin{minipage}[t]{0.3\linewidth}
\centering
\includegraphics[height=4cm, width=5.2cm]{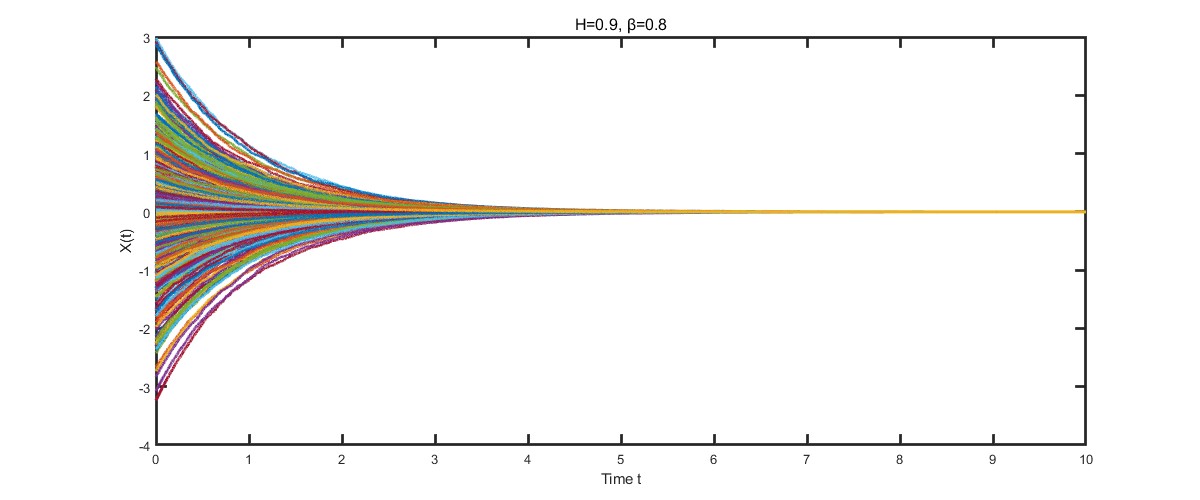}
\end{minipage}%
\centering\\
Figure 1: \begin{scriptsize} Computer simulations of the path of $X_t$ from (\ref{5.2}) are performed for $H=0.7$, $\beta=0.6$; $H=0.8$, $\beta=0.7$; and $H=0.9$, $\beta=0.8$, respectively.\end{scriptsize}
\end{figure}

\begin{figure}[htbp]
\centering
\begin{minipage}[t]{0.3\linewidth}
\centering
\includegraphics[height=4cm, width=5.2cm]{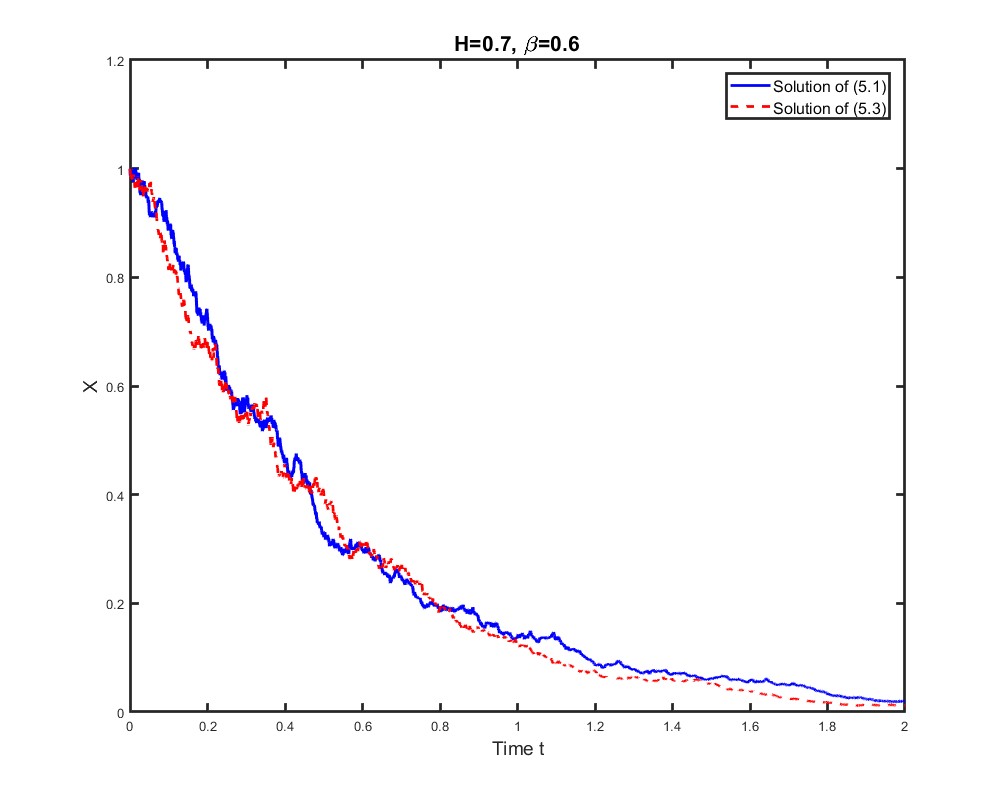}
\end{minipage}%
\begin{minipage}[t]{0.3\linewidth}
\centering
\includegraphics[height=4cm, width=5.2cm]{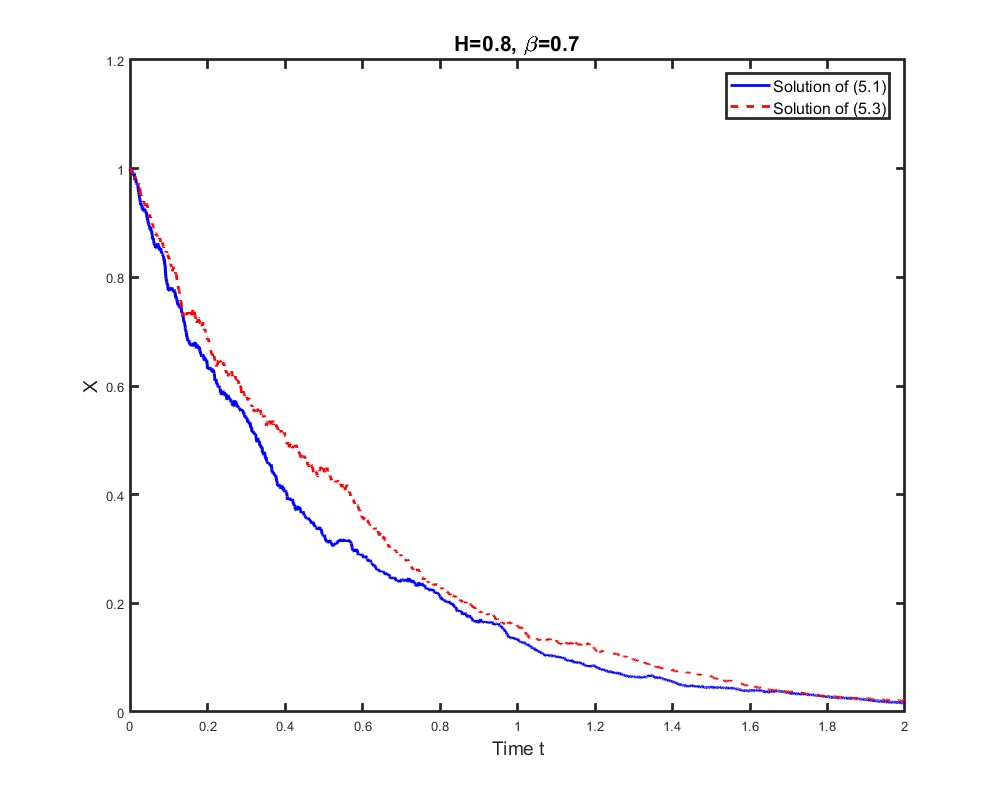}
\end{minipage}%
\begin{minipage}[t]{0.3\linewidth}
\centering
\includegraphics[height=4cm, width=5.2cm]{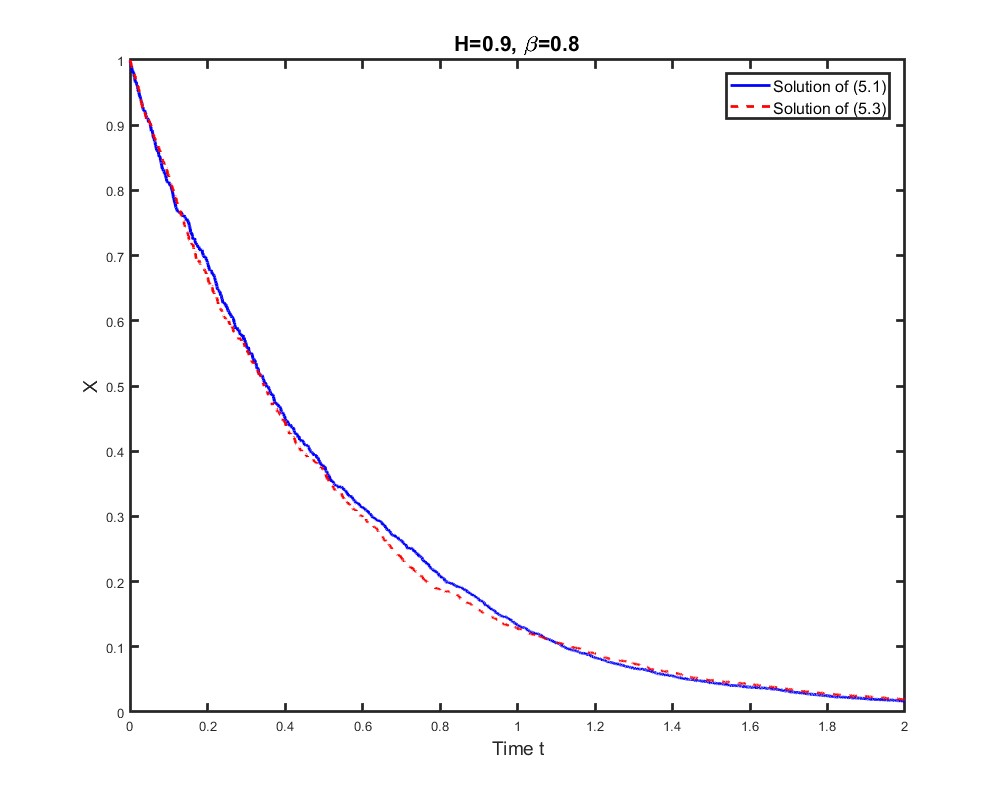}
\end{minipage}%
\centering\\
Figure 2: \begin{scriptsize} Comparisons of the solutions of (\ref{5.1}) and (\ref{5.3})  is presented  for $H=0.7$, $\beta=0.6$; $H=0.8$, $\beta=0.7$; and $H=0.9$, $\beta=0.8$, respectively, using a particle number of $N=1000$. \end{scriptsize}
\end{figure}

Figure 1 shows the trajectory of a single simulation respectively, for $H = 0.7$, $H = 0.8$, and $H = 0.9$, while the initial value $X_0$ follows a standard normal distribution. We can observe that different Hurst parameters generate different impacts on clustering. Figure 2  illustrates the convergence of the EM method under different values of $H$.

\bigskip

$\begin{array}{cc}
			\begin{minipage}[t]{1\textwidth}
				{\bf Guangjun Shen}\\
				Department of Mathematics, Anhui Normal University, Wuhu 241002, China\\
				\texttt{gjshen@ahnu.edu.cn}
			\end{minipage}
			\hfill
		\end{array}$

$\begin{array}{cc}
			\begin{minipage}[t]{1\textwidth}
				{\bf Jiangpeng Wang}\\
				Department of Mathematics, Anhui Normal University, Wuhu 241002, China\\
				\texttt{jpwangmath@163.com}
			\end{minipage}
			\hfill
		\end{array}$

$\begin{array}{cc}
			\begin{minipage}[t]{1\textwidth}
				{\bf Xuekang Zhang}\\
				School of Mathematics-Physics and Finance, Anhui Polytechnic University, Wuhu 241000, China\\
				\texttt{xkzhang@ahpu.edu.cn}
			\end{minipage}
			\hfill
		\end{array}$
\end{document}